\documentclass[a4paper]{article}
\usepackage[margin=25mm]{geometry}
\usepackage{amsmath}
\usepackage{amsthm}
\usepackage{amssymb}
\usepackage{amsfonts}
\usepackage{latexsym}
\usepackage{mathtools}
\usepackage{array}
\usepackage{bm}
\usepackage{comment} 
\usepackage{authblk}

\newtheorem{prop}{Proposition}[section]
\newtheorem{lem}[prop]{Lemma}
\newtheorem{thm}[prop]{Theorem}
\newtheorem{cor}[prop]{Corollary}

\theoremstyle{definition}
\newtheorem{dfn}[prop]{Definition}
\newtheorem{remark}[prop]{Remark}
\newtheorem{que}[prop]{Question}

\numberwithin{prop}{section}
\numberwithin{equation}{section}

\title{Completeness theorems for modal logic\\ in second-order arithmetic}
\author{Sho Shimomichi}
\author{Yuto Takeda}
\author{Keita Yokoyama}
\affil{Mathematical Institute, Tohoku University}

\usepackage[T1]{fontenc}

\begin{document}
\maketitle
%
%

%
%
%
%
\begin{abstract}
This paper investigates the logical strength of completeness theorems for modal propositional logic within second-order arithmetic. We demonstrate that the weak completeness theorem for modal propositional logic is provable in $\mathrm{RCA}_0$, and that, over $\mathrm{RCA}_0$, $\mathrm{ACA}_0$ is equivalent to the strong completeness theorem for modal propositional logic using canonical models. We also consider a simpler version of the strong completeness theorem without referring to canonical models and show that it is equivalent to $\mathrm{WKL}_0$ over $\mathrm{RCA}_0$.
\end{abstract}
\section{Introduction}
In this paper, we investigate the strength of completeness theorems for modal propositional logic with respect to Kripke semantics in second-order arithmetic.
There have been several attempts to analyze logic in second-order arithmetic. In \cite{MR2517689}, Simpson studied the relations of several well-known theorems of classical logic in second-order arithmetic and showed that, over $\mathrm{RCA}_0$, the G\"odel's completeness theorem is equivalent to $\mathrm{WKL}_0$.
Similar studies have been conducted on intuitionistic logic, and in \cite{MR2010178}, Yamazaki showed, over $\mathrm{RCA}_0$, the strong completeness theorem for intuitionistic predicate logic is equivalent to $\mathrm{ACA}_0$.
In this paper, we extend these studies to modal propositional logic.
We start with the weak completeness theorem for modal propositional logic which is provable in $\mathrm{RCA}_0$. Then we see that the strong completeness theorem by means of canonical models requires $\mathrm{ACA}_0$, while the strong completeness theorem itself is equivalent to $\mathrm{WKL}_0$ over $\mathrm{RCA}_0$.

In Section 2, we formalize modal logic in second-order arithmetic.
In Section 3, we show the weak completeness theorems for all modal logics built from the axiom $K$ using modal formulas $T$, $B$, $4$, and $D$ in $\mathrm{RCA}_0$. We formalize the proof of weak completeness by Moss \cite{MR2349876} within $\mathrm{RCA}_0$. In Section 4, we show the weak completeness theorem for $\mathrm{GL}$ in $\mathrm{RCA}_0$. In Section 5, we prove that $\mathrm{ACA}_0$ is equivalent to the strong completeness theorem using canonical models over $\mathrm{RCA}_0$. The idea of the equivalence proof depends on Yamazaki \cite{MR2010178}. In Section 6, we consider a simpler version of the strong completeness theorem. 

\renewcommand{\thefootnote}{\fnsymbol{footnote}}
\section{Modal logic in second-order arithmetic}
In this section, we set up the basic notions of modal logic within $\mathrm{RCA}_0$. As usual, we identify finite sequences of symbols of language for modal logic with natural numbers. 

Let $\mathrm{Prop}=\{p,q,r,\dots\}$ be an infinite set of propositional variables, and let $\mathrm{C}=\{\alpha_0, \alpha_1, \dots\}$ be an infinite set of propositional constants.%
\footnote{We distinguish propositional variables (may be seen as place holders) which allow uniform substitution and atomic constants. This helps finite axiomatization under the uniform substitution rule, but it does not change provability.}
Thus in $\mathrm{RCA}_0$, the set of all atomic formulas, $\mathrm{Atm}=\mathrm{Prop}\cup\mathrm{C}\cup\{\bot\}$, and the set of all modal formulas, $\mathrm{Fml}$, exist as the sets of natural numbers. 

Fix $p, q, r\in\mathrm{Prop}$. Let $\mathrm{Axm}$, the axioms of normal modal logic, be the set of the following formulas:
\begin{align*}
&p\rightarrow(q\rightarrow p),&& (p\rightarrow(q\rightarrow r))\rightarrow((p\rightarrow q)\rightarrow(p\rightarrow r)),\\
&(\neg p\rightarrow\neg q)\rightarrow(q\rightarrow p),&& \Box(p\rightarrow q)\rightarrow(\Box p\rightarrow\Box q).
\end{align*}

As in standard studies of modal logic, we also deal with the following formulas.
\begin{align*}
&T\equiv\Box p\rightarrow p,&& B\equiv p\rightarrow\Box\Diamond p,&& D\equiv \Box p\rightarrow\Diamond p, \quad\quad  4\equiv\Box p\rightarrow\Box\Box p,&\\
&5\equiv\Diamond p\rightarrow\Box\Diamond p,&& .2\equiv\Diamond\Box p\rightarrow\Box\Diamond p,&& L\equiv\Box(\Box p\rightarrow p)\rightarrow\Box p.&
\end{align*}

In dealing with modal logic, we introduce the uniform substitution rule.
Let $\mathcal{F}$ be the set of all functions from $\mathrm{Prop}$ to $\mathrm{Fml}$ with finite domains.

\begin{prop}[$\mathrm{RCA}_0$]\label{uniform}
There exisits a function $f:\mathcal{F}\times\mathrm{Fml}\rightarrow\mathrm{Fml}$ such that for all finite function $\sigma:\subseteq\mathrm{Prop}\rightarrow\mathrm{Fml}$, $f(\sigma, \cdot)$ is the \emph{uniform substitution for $\sigma$}, i.e.,
\begin{itemize}
    \item $f(\sigma,p)=\sigma(p)$ if $p\in\mathrm{Prop}$ and $p\in\mathrm{dom}(\sigma)$,
    \item $f(\sigma,\alpha)=\alpha$ if $\alpha\in\mathrm{C}$,
    \item $f(\sigma,\bot)=\bot$,
    \item $f(\sigma,\varphi\rightarrow\psi)=f(\sigma,\varphi)\rightarrow f(\sigma,\psi)$,
    \item $f(\sigma,\Box\varphi)=\Box f(\sigma,\varphi)$.
\end{itemize}
\end{prop}

\begin{proof}
We define $f:\mathcal{F}\times\mathrm{Fml}\rightarrow\mathrm{Fml}$ recursively: for all $p\in\mathrm{Prop}\cup\mathrm{C}$ and $\sigma\in\mathcal{F}$, $f(\sigma, p)=\sigma(p)$ if $p\in\mathrm{Prop}\land p\in\mathrm{dom}(\sigma)$, otherwise $f(\sigma, p)=p$. This is possible using $\mathrm{I}\mathrm{\Sigma}^0_1$.
\end{proof}

Henceforth, for all $\varphi\in\mathrm{Fml}$ and $\sigma\in\mathcal{F}$, we denote $f(\sigma,\varphi)$ as $\overline{\sigma}(\varphi)$. Note that for all $\sigma\in\mathrm{Seq}$, we denote the length of $\sigma$ as $lh(\sigma)$.
We now formalize the notion of provability.
\begin{dfn}[$\mathrm{RCA}_0$]
Let $\mathrm{\mathrm{\Gamma}}\subset\mathrm{Fml}$. We will define the following predicates:
\begin{equation*}
\begin{split}
\mathrm{Prf}_{\mathbf{K}\mathrm{\mathrm{\Sigma}}}(\mathrm{\mathrm{\Gamma}}, p)\equiv p&\in\mathrm{Seq}\land\forall k(k<lh(p)\rightarrow p(k)\in\mathrm{Fml})\\&\land\forall k\Bigl(k<lh(p)\rightarrow\Bigl(p(k)\in\mathrm{\mathrm{\Gamma}}\lor p(k)\in\mathrm{\mathrm{\Sigma}}\lor p(k)\in\mathrm{Axm}\\&\lor(\exists i<k\exists j<k)(p(i)=p(j)\rightarrow p(k))\lor(\exists i<k)(p(k)=\Box p(i))\\&\lor(\exists i<k)(\exists\sigma:\{q<p(i)\mid q\in\mathrm{Prop}\}\rightarrow\{\varphi<p(k)\mid\varphi\in\mathrm{Fml}\})\\&(p(k)=\overline{\sigma}(p(i)))\Bigr)\Bigr),
\end{split}
\end{equation*}
\begin{equation*}
\begin{split}
\mathrm{Pbl}_{\mathbf{K}\mathrm{\mathrm{\Sigma}}}(\mathrm{\mathrm{\Gamma}}, \varphi)\equiv&\exists\psi_1,\dots,\exists\psi_n\in\mathrm{\mathrm{\Gamma}}\exists p\\&\bigl(\mathrm{Prf}_{\mathbf{K}\mathrm{\mathrm{\Sigma}}}(\emptyset, p)\land(\exists i<lh(p))(p(i)=(\psi_1\land\dots\land\psi_n\rightarrow\varphi))\bigr).
\end{split}
\end{equation*}
\end{dfn}

Let $\mathrm{\mathrm{\Gamma}}\subset\mathrm{Fml}$. Then,
\begin{itemize}
    \item $\mathrm{\mathrm{\Gamma}}$ is \emph{$\mathbf{K}\mathrm{\mathrm{\Sigma}}$-consistent} if $\neg\mathrm{Pbl}_{\mathbf{K}\mathrm{\mathrm{\Sigma}}}(\mathrm{\mathrm{\Gamma}}, \bot)$ holds,
    \item $\mathrm{\mathrm{\Gamma}}$ is \emph{closed under deduction} if $\forall\varphi\in\mathrm{Fml}(\mathrm{Pbl}_{\mathbf{K}\mathrm{\mathrm{\Sigma}}}(\mathrm{\mathrm{\Gamma}}, \varphi)\rightarrow\varphi\in\mathrm{\Gamma})$ holds,
    \item $\mathrm{\mathrm{\Gamma}}$ is \emph{complete} if $\forall\varphi\in\mathrm{Fml}(\varphi\in \mathrm{\mathrm{\Gamma}}\lor\varphi\not\in\mathrm{\mathrm{\Gamma}})$ holds.
    \item $\mathrm{\mathrm{\Gamma}}$ is \emph{maximally $\mathbf{K}\mathrm{\mathrm{\Sigma}}$-consistent} if $\mathrm{\mathrm{\Gamma}}$ is $\mathbf{K}\mathrm{\mathrm{\Sigma}}$-consistent, closed under deduction, and complete.
\end{itemize}

We next formalize the Kripke semantics.
\begin{dfn}[$\mathrm{RCA}_0$]
A \emph{Kripke model} is a tuple $M=(W, R, V)$ satisfying the following conditions:
\begin{enumerate}
    \item $W\subseteq\mathbb{N}$ is a non-empty set,
    \item $R$ is a binary relation on $W$, i.e., $R\subseteq W\times W$,
    \item 
$V : W\times\mathrm{Fml} \rightarrow\{0, 1\}$,
    \item $V(w, \bot)=0$ for any $w\in W$, 
    \item $V(w, \varphi\rightarrow\psi)=1-V(w, \varphi)(1-V(w, \psi))$ for any $w\in W$ and any $\varphi, \psi\in\mathrm{Fml}$, 
    \item $V(w, \Box\varphi)=1\iff\forall v\in W(wRv\rightarrow V(v, \varphi)=1)$ for any $w\in W$ and any $\varphi\in\mathrm{Fml}$.
\end{enumerate}
The pair $(W,R)$ is called a \emph{frame} and the function $V$ is called a \emph{valuation} with $(W, R)$.
\end{dfn}

We also use the following notations.
Let $F$ be a frame, $M$ be a model, $w\in W$, and $\varphi\in\mathrm{Fml}$, we define
\begin{align*}
&M, w\Vdash\varphi\equiv V(w, \varphi)=1, \quad \quad M\Vdash\varphi\equiv(\forall w\in W)(V(w, \varphi)=1),\\
&F\Vdash\varphi\equiv(\forall V\text{: valuation with $F$})(\forall w\in W)(V(w, \varphi)=1).
\end{align*}
\begin{remark}
In the above definition, the valuation $V$ should cover the truth values of all formulas. Within $\mathrm{RCA}_{0}$, one cannot extend the valuation for atomic formulas to the full valuation.
\end{remark}

\begin{dfn}[$\mathrm{RCA}_0$]
Let $F=(W, R)$ be a frame.
\begin{itemize}
    \item $F$ is \emph{appropriate to $T$} if $R$ is reflexive,
    \item $F$ is \emph{appropriate to $B$} if $R$ is symmetric,
    \item $F$ is \emph{appropriate to $4$} if $R$ is transitive,
    \item $F$ is \emph{appropriate to $5$} if $R$ is Euclidean,
    \item $F$ is \emph{appropriate to $D$} if $R$ is serial,
    \item $F$ is \emph{appropriate to $.2$} if $R$ is directed,
    \item $F$ is \emph{appropriate to $L$} if $R$ is transitive and $W$ has no infinite assending sequences by $R$.
\end{itemize}

In general, $F$ is \emph{appropriate to $\mathrm{\Sigma}\subseteq\{T, B, 4, 5, D, .2, L\}$} if for all $\varphi\in\mathrm{\Sigma}$, $F$ is appropriate to $\varphi$.
\end{dfn}
Given a frame $F$, a set of formulas $\mathrm{\Gamma}$, $\mathrm{\Sigma}\subseteq\{T, B, 4, 5, D, .2, L\}$, and $\varphi\in\mathrm{Fml}$, we let
\begin{align*}
F\Vdash\mathrm{\Gamma}&\equiv(\forall\psi\in\mathrm{\Gamma})(F\Vdash\psi),\\
\mathrm{\Gamma}\Vdash_{\mathbb{F}_{\mathrm{\Sigma}}}\varphi&\equiv(\forall F\text{: frame which is appropriate to $\mathrm{\Sigma}$})(F\Vdash\mathrm{\Gamma}\rightarrow F\Vdash\varphi).
\end{align*}

Now we are ready to state the soundness theorem.
\begin{thm}[Soundness Theorem, $\mathrm{RCA}_0$]\label{sound}
Let $\mathrm{\Sigma}\subseteq\{T, B, 4, 5, D, .2, L\}$ and $\varphi\in\mathrm{Fml}$. If $\mathrm{Pbl}_{\mathbf{K}\mathrm{\mathrm{\Sigma}}}(\emptyset, \varphi)$ holds, then $\Vdash_{\mathbb{F}_{\mathrm{\Sigma}}}\varphi$ holds.
\end{thm}
To prove the soundness theorem within $\mathrm{RCA}_{0}$, we carefully formalize the usual proof so that the induction fits in $\mathrm{I}\mathrm{\Sigma}^0_1$.

\begin{proof}
Let $p$ be a proof of $\varphi$, that is, $\mathrm{Prf}_{\mathbf{K}\mathrm{\Sigma}}(\emptyset, p)\land(\exists k<lh(p))(p(k)=\varphi)$ holds. Fix a frame $F$ which is appropriate to $\mathrm{\Sigma}$ and a valuation $V$ with $F$. Let $M=(F, V)$. We want to show that $(\forall k<lh(p))(\forall w\in W)(V(w, p(k))=1)$. However, this way is difficult due to the complexity of induction for the uniform substitution rule. Thus we aim to show the following by $\mathrm{I}\mathrm{\Sigma}^0_1$:
\begin{equation*}
(\forall k<lh(p))(\forall w\in W)(\forall\sigma:\subseteq\mathrm{Prop}\rightarrow\mathrm{Fml})(V(w, \overline{\sigma}(p(k)))=1).
\end{equation*}

Then it is enough to show that
\begin{enumerate}
    \item $\mathrm{\Sigma}$ is satisfiable for $M$ \\(i.e., $(\forall\varphi\in\mathrm{\Sigma})(\forall w\in W)(\forall\sigma:\subseteq\mathrm{Prop}\rightarrow\mathrm{Fml})(V(w, \overline{\sigma}(\varphi))=1)$),
    \item All axioms are satisfiable for $M$ \\(i.e., $(\forall\varphi\in\mathrm{Axm})(\forall w\in W)(\forall\sigma:\subseteq\mathrm{Prop}\rightarrow\mathrm{Fml})(V(w, \overline{\sigma}(\varphi))=1)$), and
    \item All inference rules are satisfiable for $M$.
\end{enumerate}

First, we show (1). We show it for $T\in\mathrm{\Sigma}$. Thus $F$ is appropriate to $T$. Let $w\in W$ and $\sigma:\subseteq\mathrm{Prop}\rightarrow\mathrm{Fml}$. Assume that $M, w\Vdash\Box\sigma(p)$. Since $F$ is reflexive, $wRw$, so $M, w\Vdash\sigma(p)$. Thus $M, w\Vdash\Box\sigma(p)\rightarrow\sigma(p)$.

We show it holds for $B\in\mathrm{\Sigma}$. Thus $F$ is appropriate to $B$. Let $w\in W$ and $\sigma:\subseteq\mathrm{Prop}\rightarrow\mathrm{Fml}$. Assume that $M, w\Vdash\sigma(p)$. Let $v\in W$ be $wRv$. Since $F$ is symmetric, $vRw$, so $M, v\Vdash\Diamond\sigma(p)$. Thus $M, w\Vdash\Box\Diamond\sigma(p)$, so $M, w\Vdash\sigma(p)\rightarrow\Box\Diamond\sigma(p)$ holds.

We show it holds for $4\in\mathrm{\Sigma}$. Thus $F$ is appropriate to $4$. Let $w\in W$ and $\sigma:\subseteq\mathrm{Prop}\rightarrow\mathrm{Fml}$. Assume that $M, w\Vdash\Box\sigma(p)$. Take $v, u\in W$ such that $wRv$ and $vRu$. Since $F$ is transitive, $wRu$, then $M, u\Vdash\sigma(p)$. Thus $M, v\Vdash\Box\sigma(p)$, so $M, w\Vdash\Box\Box\sigma(p)$ holds, so $M, w\Vdash\Box\sigma(p)\rightarrow\Box\Box\sigma(p)$ holds.

We show it holds for $5\in\mathrm{\Sigma}$. Thus $F$ is appropriate to $5$. Let $w\in W$ and $\sigma:\subseteq\mathrm{Prop}\rightarrow\mathrm{Fml}$. Assume that $M, w\Vdash\Diamond\sigma(p)$. Fix $u\in W$ with $wRu$. There exists $v\in W$ such that $wRv$ and $M, v\Vdash\sigma(p)$. Since $F$ is Euculidean, $uRv$, so $M, u\Vdash\Diamond\sigma(p)$. Thus $M, w\Vdash\Box\Diamond\sigma(p)$, so $M, w\Vdash\Diamond\sigma(p)\rightarrow\Box\Diamond\sigma(p)$ holds.

We show it holds for $D\in\mathrm{\Sigma}$. Thus $F$ is appropriate to $D$. Let $w\in W$ and $\sigma:\subseteq\mathrm{Prop}\rightarrow\mathrm{Fml}$. Assume that $M, w\Vdash\Box\sigma(p)$. Since $F$ is serial, there exists $u\in W$ such that $wRu$, so $M, u\Vdash\sigma(p)$. Thus $M, w\Vdash\Diamond\sigma(p)$ holds.

We show it holds for $.2\in\mathrm{\Sigma}$. Thus $F$ is appropriate to $.2$. Let $w\in W$ and $\sigma:\subseteq\mathrm{Prop}\rightarrow\mathrm{Fml}$. Assume that $M, w\Vdash\Diamond\Box\sigma(p)$. Fix $u\in W$ with $wRu$. There exists $v\in W$ such that $wRv$ and $M, v\Vdash\Box\sigma(p)$. Since $F$ is directed, there exists $s\in W$ such that $uRs$ and $vRs$, so $M, s\Vdash\sigma(p)$. Thus $M, u\Vdash\Diamond\sigma(p)$, so $M, w\Vdash\Box\Diamond\sigma(p)$. Thus $M, w\Vdash\Diamond\Box\sigma(p)\rightarrow\Box\Diamond\sigma(p)$ holds.

We show it holds for $L\in\mathrm{\Sigma}$. Thus $F$ is appropriate to $L$. Let $w\in W$ and $\sigma:\subseteq\mathrm{Prop}\rightarrow\mathrm{Fml}$. Assume that $M, w\Vdash\Box(\Box\sigma(p)\rightarrow\sigma(p))$ and $M, w\not\Vdash\Box\sigma(p)$. Then we can define $w_0, w_1, \dots$ by the following construction recursively: by $M, w\not\Vdash\Box\sigma(p)$, we can take $w_0\in W$ such that $wRw_0$ and $M, w_0\not\Vdash\sigma(p)$. Assume that $w_n$ is defined, that is, $M, w_n\not\Vdash\sigma(p)$ and for all $i<j\leq n$, $w_iRw_j$, then by $R$ is transitive, $wRw_n$. Thus $M, w_n\Vdash\Box\sigma(p)\rightarrow\sigma(p)$, and since $M, w_n\not\Vdash\sigma(p)$, $M, w_n\not\Vdash\Box\sigma(p)$. Thus we can take $w_{n+1}$ such that $w_nRw_{n+1}$ and $M, w_{n+1}\not\Vdash\sigma(p)$. This constraction is possible using $\mathrm{I}\mathrm{\Sigma}^0_1$. Then we heve an infinite assending sequence by $R$. It is a contradiction. Therefore $M, w\Vdash\Box(\Box \sigma(p)\rightarrow\sigma(p))\rightarrow\Box\sigma(p)$.

We show (2). Clearly, the tautologies of modal logic are satisfiable for $M$. We show that $K$ is satisfiable for $M$. Let $w\in W$ and $\sigma:\subseteq\mathrm{Prop}\rightarrow\mathrm{Fml}$. Suppose that $M, w\Vdash\Box(\sigma(p)\rightarrow\sigma(q))$ and $M, w\Vdash\Box\sigma(p)$ holds. Fix $v\in W$ with $wRv$, then $M, v\Vdash\sigma(p)\rightarrow\sigma(q)$ and $M, v\Vdash\sigma(p)$ holds. Thus $M, v\Vdash\sigma(q)$, so $M, w\Vdash\Box\sigma(q)$ holds. Thus $M, w\Vdash\Box(\sigma(p)\rightarrow\sigma(q))\rightarrow(\Box\sigma(p)\rightarrow\Box\sigma(q))$.

We show (3). First, we show that Modus Ponens is satisfiable for $M$. Let $\varphi_0, \psi_0\in\mathrm{Fml}$, Assume that for all $w, v\in W$, and $\sigma, \tau:\subseteq\mathrm{Prop}\rightarrow\mathrm{Fml}$, $M, w\Vdash\overline{\sigma}(\varphi_0)$ and $M, v\Vdash\overline{\tau}(\varphi_0\rightarrow\psi_0)$ holds. Let $w_0\in W$, and $\sigma_0:\subseteq\mathrm{Prop}\rightarrow\mathrm{Fml}$. Then $M, w_0\Vdash\overline{\sigma_0}(\varphi_0)$ and $M, w_0\Vdash\overline{\sigma_0}(\varphi_0\rightarrow\psi_0)$ holds, thus $M, w_0\Vdash\overline{\sigma_0}(\varphi_0)$ and $M, w_0\Vdash\overline{\sigma_0}(\varphi_0)\rightarrow\overline{\sigma_0}(\psi_0)$ holds, thus $M, w_0\Vdash\overline{\sigma_0}(\psi_0)$ holds. 

Second, we show that Necessitation is satisfiable for $M$. Let $\varphi_0\in\mathrm{Fml}$. Assume that for all $w\in W$, and $\sigma:\subseteq\mathrm{Prop}\rightarrow\mathrm{Fml}$, $M, w\Vdash\overline{\sigma}(\varphi_0)$ holds. Let $w_0\in W$, and $\sigma_0:\subseteq\mathrm{Prop}\rightarrow\mathrm{Fml}$. Then $M, w_0\Vdash\overline{\sigma_0}(\varphi_0)$ holds, thus $M, w_0\Vdash\Box\overline{\sigma_0}(\varphi_0)$ holds, so $M, w_0\Vdash\overline{\sigma_0}(\Box\varphi_0)$ holds. 

Third, we show that uniform substitution is satisfiable for $M$. Let $\varphi_0\in\mathrm{Fml}$. Assume that for all $w\in W$, and $\sigma:\subseteq\mathrm{Prop}\rightarrow\mathrm{Fml}$, $M, w\Vdash\overline{\sigma}(\varphi_0)$. Let $w_0\in W$, and $\sigma_0:\subseteq\mathrm{Prop}\rightarrow\mathrm{Fml}$. Let $\sigma_1:\subseteq\mathrm{Prop}\rightarrow\mathrm{Fml}$. Then $M, w_0\Vdash\overline{\sigma_0}(\varphi_0)$ holds. Moreover, $M, w_0\Vdash\overline{\sigma_1}(\varphi_0)$ holds. Then by the definition of uniform substitution, $M, w_0\Vdash\overline{\sigma_0}(\overline{\sigma_1}(\varphi_0))$.  
\end{proof}

\begin{cor}[$\mathrm{RCA}_0$]
For all $\mathrm{\Sigma}\subseteq\{T, B, 4, 5, D, .2\}$, $\mathbf{K}\mathrm{\mathrm{\Sigma}}$ is consistent, i.e., $\neg\mathrm{Pbl}_{\mathbf{K}\mathrm{\mathrm{\Sigma}}}(\emptyset, \bot)$ holds. Moreover, $\mathbf{GL}$ is consistent, \\i.e., $\neg\mathrm{Pbl}_{\mathbf{GL}}(\emptyset, \bot)$ holds.
\end{cor}

\begin{proof}
By Theorem~\ref{sound}, it is enough to show that there exists a frame $F$ which is appropriate to $\mathrm{\Sigma}$. $F=(\{0\}, \{(0, 0)\})$ is the desired frame. Moreover, $F=(\{0\}, \emptyset)$ is a frame which is appropriate to $L$.
\end{proof}

\section{Weak Completeness Theorem in second-order arithmetic}
In this section, we will prove the weak completeness theorems for modal logic in $\mathrm{RCA}_0$. For this purpose, we will formalize the proof of weak completeness theorem by Moss \cite{MR2349876} in $\mathrm{RCA}_0$. Unless otherwise specified, we will assume $\mathrm{\Sigma}\subseteq\{T, B, 4, D\}$. First, for a given $\mathbf{K}\mathrm{\mathrm{\Sigma}}$-consistent formula, we aim to create the weak Kripke model in $\mathrm{RCA}_0$. 

For any $\varphi\in\mathrm{Fml}$, $sub(\varphi)$ is the set of subformulas of $\varphi$. A weak Kripke model for $\varphi$ is a Kripke model with the valuation restricted to subformulas of $\varphi$.

\begin{dfn}[Weak Kripke model, $\mathrm{RCA}_0$]
\emph{Weak Kripke model of $\varphi$} is a tuple $M=(W, R, V)$ satisfying the following conditions:
\begin{enumerate}
    \item $W\subseteq\mathbb{N}$ is a non-empty set,
    \item $R$ is a binary relation on $W$, i.e., $R\subseteq W\times W$,
    \item $V : W\times sub(\varphi) \rightarrow\{0, 1\}$,
    \item $V(w, \bot)=0$ for any $w\in W$, 
    \item $V(w, \psi\rightarrow\theta)=1-V(w, \psi)(1-V(w, \theta))$ for any $w\in W$ and any $\psi, \theta\in sub(\varphi)$, 
    \item $V(w, \Box\psi)=1\iff\forall v\in W(wRv\rightarrow V(v, \psi)=1)$ for any $w\in W$ and any $\psi\in sub(\varphi)$.
\end{enumerate}
\end{dfn}

To construct the model, we make some preparations.
Let $p_0, p_1, \dots$ be an enumeration of $\mathrm{Atm}$.
We define the \emph{height fuction} $ht : \mathrm{Fml}\rightarrow\mathbb{N}$ and the \emph{order function} $ord : \mathrm{Fml}\rightarrow\mathbb{N}$ by primitive recursion:
\begin{itemize}
    \item $ht(p_n)=0$ for all $n$,
    \item $ht(\varphi\rightarrow\psi)=\max\{ht(\varphi), ht(\psi)\}$,
    \item $ht(\Box\varphi)=1+ht(\varphi)$.
    \item $ord(p_n)=n$ for all $n$,
    \item $ord(\varphi\rightarrow\psi)=\max\{ord(\varphi), ord(\psi)\}$,
    \item $ord(\Box\varphi)=ord(\varphi)$.
\end{itemize}

\begin{dfn}[$\mathrm{RCA}_0$]
We define the followings.
\begin{itemize}
    \item $\mathcal{L}_{h, n}=\{\varphi\in\mathrm{Fml}\mid ht(\varphi)\leq h\land ord(\varphi)\leq n\}$ for all $h$ and $n$,
    \item $\hat{T}\equiv\bigwedge T\land\bigwedge_{p\in\{p_0, \dots, p_n\}\setminus T}\neg p$  for all $T\subseteq\{p_0, \dots, p_n\}$,
    \item $\bigoplus X\equiv\bigvee_{\alpha\in X}(\alpha\land\bigwedge_{\beta\in X\setminus\{\alpha\}}\neg\beta)$ for all $X\subseteq\mathrm{Fml}$.
\end{itemize}

For any $n$, we define the finite set $C_{h, n}\subseteq\mathrm{Fml}$ recursively:
\begin{itemize}
    \item $C_{0, n}=\{\hat{T}\mid T\subseteq\{p_0, \dots, p_n\}\}$,
    \item $C_{h+1, n}=\{\alpha_{S, T}\equiv\bigl(\bigwedge_{\psi\in S}\Diamond\psi\bigr)\land\bigl(\Box\bigvee S\bigr)\land\hat{T}\mid S\subseteq C_{h, n}, T\subseteq\{p_0, \dots, p_n\}\}$.
\end{itemize}
\end{dfn}

The members of $C_{h, n}$ are called \emph{canonical formulas}. They are of height $\leq h$ and built from the first $n$ atomic formulas.
Note that each $C_{h, n}$ exists as a finite set and they are uniformly available within $\mathrm{RCA}_{0}$ by the primitive recursion.
We see several basic lemmas.
\begin{lem}[$\mathrm{RCA}_0$]\label{select}
Fix $h$ and $n$. For all $\varphi\in\mathcal{L}_{h, n}$ and for all $\alpha\in C_{h, n}$, either $\mathrm{Pbl}_{\mathbf{K}\mathrm{\mathrm{\Sigma}}}(\emptyset, \alpha\rightarrow\varphi)$ or $\mathrm{Pbl}_{\mathbf{K}\mathrm{\mathrm{\Sigma}}}(\emptyset, \alpha\rightarrow\neg\varphi)$ holds.
\end{lem}

\begin{proof}
Let $\alpha\equiv\alpha_{S, T}\in C_{h, n}$. We show this lemma by induction on the length of $\varphi$ using $\mathrm{I}\mathrm{\Sigma}^0_1$.

Assume that $\varphi$ is a propositional variable $p$. If $p\in T$, then $\mathrm{Pbl}_{\mathbf{K}\mathrm{\mathrm{\Sigma}}}(\emptyset, \hat{T}\rightarrow p)$ holds. If $p\not\in T$, then $\mathrm{Pbl}_{\mathbf{K}\mathrm{\mathrm{\Sigma}}}(\emptyset, \hat{T}\rightarrow\neg p)$ holds.  

Assume that $\varphi\equiv\bot$. Then $\mathrm{Pbl}_{\mathbf{K}\mathrm{\mathrm{\Sigma}}}(\emptyset, \hat{T}\rightarrow\neg\bot)$ holds.

Assume that $\varphi\equiv\psi\rightarrow\theta$. Then by the assumption of induction, one of (1) $\mathrm{Pbl}_{\mathbf{K}\mathrm{\mathrm{\Sigma}}}(\emptyset, \alpha\rightarrow\neg\psi)$ or (2) $\mathrm{Pbl}_{\mathbf{K}\mathrm{\mathrm{\Sigma}}}(\emptyset, \alpha\rightarrow\theta)$ or (3) $\mathrm{Pbl}_{\mathbf{K}\mathrm{\mathrm{\Sigma}}}(\emptyset, \alpha\rightarrow\psi)\land\mathrm{Pbl}_{\mathbf{K}\mathrm{\mathrm{\Sigma}}}(\emptyset, \alpha\rightarrow\neg\theta)$ is satisfiable. If either (1) or (2) holds, then $\mathrm{Pbl}_{\mathbf{K}\mathrm{\mathrm{\Sigma}}}(\emptyset, \alpha\rightarrow\psi\rightarrow\theta)$ holds. If (3) holds, then $\mathrm{Pbl}_{\mathbf{K}\mathrm{\mathrm{\Sigma}}}(\emptyset, \alpha\rightarrow\neg(\psi\rightarrow\theta))$ holds.

Assume that $\varphi\equiv\Box\psi$. Then the assumption of induction, for all $\beta$, either $\mathrm{Pbl}_{\mathbf{K}\mathrm{\mathrm{\Sigma}}}(\emptyset, \beta\rightarrow\psi)$ or $\mathrm{Pbl}_{\mathbf{K}\mathrm{\mathrm{\Sigma}}}(\emptyset, \beta\rightarrow\neg\psi)$ holds. We have two cases. First, assume that for some $\beta_0\in S$, $\mathrm{Pbl}_{\mathbf{K}\mathrm{\mathrm{\Sigma}}}(\emptyset, \beta_0\rightarrow\neg\psi)$. Then $\mathrm{Pbl}_{\mathbf{K}\mathrm{\mathrm{\Sigma}}}(\emptyset, \Diamond\beta_0\rightarrow\Diamond\neg\psi)$ holds. Thus $\mathrm{Pbl}_{\mathbf{K}\mathrm{\mathrm{\Sigma}}}(\emptyset, \alpha\rightarrow\neg\Box\psi)$ holds. The other case is when for each $\beta\in S$, $\mathrm{Pbl}_{\mathbf{K}\mathrm{\mathrm{\Sigma}}}(\emptyset, \beta\rightarrow\psi)$ holds. Then $\mathrm{Pbl}_{\mathbf{K}\mathrm{\mathrm{\Sigma}}}(\emptyset, \bigvee S\rightarrow\psi)$ holds. Thus $\mathrm{Pbl}_{\mathbf{K}\mathrm{\mathrm{\Sigma}}}(\emptyset, \Box(\bigvee S)\rightarrow\Box\psi)$holds, so $\mathrm{Pbl}_{\mathbf{K}\mathrm{\mathrm{\Sigma}}}(\emptyset, \alpha\rightarrow\Box\psi)$ holds.
\end{proof}

\begin{prop}[$\mathrm{RCA}_0$]\label{bigo}
For all $h$ and $n$, $\mathrm{Pbl}_{\mathbf{K}\mathrm{\mathrm{\Sigma}}}(\emptyset, \bigoplus C_{h, n})$ holds.
\end{prop}
\begin{proof}
First, we show that for all $n$, $\mathrm{Pbl}_{\mathbf{K}\mathrm{\mathrm{\Sigma}}}(\emptyset, \bigoplus C_{0, n})$ by $\mathrm{I}\mathrm{\Sigma}^0_1$. $\bigoplus C_{0, 0}=p_0\lor\neg p_0$. Thus $\mathrm{Pbl}_{\mathbf{K}\mathrm{\mathrm{\Sigma}}}(\emptyset, \bigoplus C_{0, 0})$. Suppose that $\mathrm{Pbl}_{\mathbf{K}\mathrm{\mathrm{\Sigma}}}(\emptyset, \bigoplus C_{0, n})$ holds. $C_{0, n+1}=\{p_{n+1}\land\hat{T}\mid T\subseteq\{p_0, \dots, p_n\}\}\cup\{\neg p_{n+1}\land\hat{T}\mid T\subseteq\{p_0, \dots, p_n\}\}$, so $\bigoplus C_{0, n+1}=(p_{n+1}\land\bigoplus C_{0, n})\lor(\neg p_{n+1}\land\bigoplus C_{0, n})$. Thus $\bigoplus C_{0, n+1}=(p_{n+1}\lor\neg p_{n+1})\land\bigoplus C_{0, n}$. By the assumption, $\mathrm{Pbl}_{\mathbf{K}\mathrm{\mathrm{\Sigma}}}(\emptyset, \bigoplus C_{0, n+1})$ holds. 

Second, fix $n$ and we show that for all $h$, $\mathrm{Pbl}_{\mathbf{K}\mathrm{\mathrm{\Sigma}}}(\emptyset, \bigoplus C_{h, n})$ holds by $\mathrm{I}\mathrm{\Sigma}^0_1$. Let $\alpha_S=\bigwedge_{\psi\in S}\Diamond\psi\land\bigwedge_{\psi\not\in S}\neg\Diamond\psi$ and $S_0\equiv\bigwedge_{\psi\in S}\Diamond\psi\land\Box\bigvee S$. $\mathrm{Pbl}_{\mathbf{K}\mathrm{\mathrm{\Sigma}}}(\emptyset, \bigoplus C_{0, n})$ holds by the above proof. Suppose that $\mathrm{Pbl}_{\mathbf{K}\mathrm{\mathrm{\Sigma}}}(\emptyset, \bigoplus C_{h, n})$ holds. Then the following holds.
\begin{equation*}
\mathrm{Pbl}_{\mathbf{K}\mathrm{\mathrm{\Sigma}}}\Bigl(\emptyset, \bigoplus C_{h+1, n}\leftrightarrow(\bigoplus\{\hat{T}\mid T\subseteq\{p_0, \dots, p_n\}\}\land\bigoplus\{S_0\mid S\subseteq C_{h, n}\})\Bigr).
\end{equation*}
Then $\mathrm{Pbl}_{\mathbf{K}\mathrm{\mathrm{\Sigma}}}(\emptyset, \bigoplus\{\alpha_{S}\mid S\subseteq C_{h, n}\}))$ holds, so it is enough to show that \\$\mathrm{Pbl}_{\mathbf{K}\mathrm{\mathrm{\Sigma}}}(\emptyset, \bigwedge_{\psi\not\in S}\neg\Diamond\psi\leftrightarrow\Box\bigvee S)$ holds for all $S\subseteq C_{h, n}$. Fix $S\subseteq C_{h, n}$. Since $\mathrm{Pbl}_{\mathbf{K}\mathrm{\mathrm{\Sigma}}}(\emptyset, \bigoplus C_{h, n})$ holds, $\mathrm{Pbl}_{\mathbf{K}\mathrm{\mathrm{\Sigma}}}(\emptyset, \bigwedge_{\psi\not\in S}\neg\psi\leftrightarrow\bigvee S)$ holds. Thus by Necessitation, $\mathrm{Pbl}_{\mathbf{K}\mathrm{\mathrm{\Sigma}}}(\emptyset, \bigwedge_{\psi\not\in S}\neg\Diamond\psi\leftrightarrow\Box\bigvee S)$ holds. 
\end{proof}

\begin{lem}[$\mathrm{RCA}_0$]\label{notempty}
Let $\varphi\in\mathcal{L}_{h, n}$ and $\psi\in\mathrm{Fml}$. If $\varphi$ is $\mathbf{K}\mathrm{\mathrm{\Sigma}}$-consistent, then there exists $\alpha\in C_{h+1, n}$ such that $\alpha$ is $\mathbf{K}\mathrm{\mathrm{\Sigma}}$-consistent and $\mathrm{Pbl}_{\mathbf{K}\mathrm{\mathrm{\Sigma}}}(\emptyset, \alpha\rightarrow\varphi)$ holds.
\end{lem}

\begin{proof}
Note that $\varphi\in\mathcal{L}_{h+1, n}$. Let $S=\{\alpha\in C_{h+1, n}\mid\mathrm{Pbl}_{\mathbf{K}\mathrm{\mathrm{\Sigma}}}(\emptyset, \alpha\rightarrow\neg\varphi)\}$.  First, we show the following,
\begin{equation*}
\mathrm{Pbl}_{\mathbf{K}\mathrm{\mathrm{\Sigma}}}(\emptyset, \neg\varphi\leftrightarrow\bigvee_{\alpha\in S}(\alpha\land\neg\varphi)\leftrightarrow\bigvee S).
\end{equation*}

Clearly, $\mathrm{Pbl}_{\mathbf{K}\mathrm{\mathrm{\Sigma}}}(\emptyset, \bigvee_{\alpha\in S}(\alpha\land\neg\varphi)\rightarrow\bigvee S\rightarrow\neg\varphi)$ holds. And for $\alpha\in C_{h, n}$ such that $\alpha\not\in S$, by Lemma~\ref{select}, $\mathrm{Pbl}_{\mathbf{K}\mathrm{\mathrm{\Sigma}}}(\emptyset, \alpha\rightarrow\varphi)$ holds. Thus $\mathrm{Pbl}_{\mathbf{K}\mathrm{\mathrm{\Sigma}}}(\emptyset, \neg\varphi\rightarrow\bigvee_{\alpha\in S}(\alpha\land\neg\varphi))$ holds. 

Assume that for all $\alpha\in C_{h+1, n}$, $\neg\mathrm{Pbl}_{\mathbf{K}\mathrm{\mathrm{\Sigma}}}(\emptyset, \alpha\rightarrow\varphi)$, that is, $\mathrm{Pbl}_{\mathbf{K}\mathrm{\mathrm{\Sigma}}}(\emptyset, \alpha\rightarrow\neg\varphi)$ holds. Then $S=C_{h+1, n}$. Since $\mathrm{Pbl}_{\mathbf{K}\mathrm{\mathrm{\Sigma}}}(\emptyset, \neg\varphi\leftrightarrow\bigvee S)$ holds, \\$\mathrm{Pbl}_{\mathbf{K}\mathrm{\mathrm{\Sigma}}}(\emptyset, \bigvee C_{h+1, n}\rightarrow\neg\varphi)$ holds. By Lemma~\ref{bigo} and $\mathrm{Pbl}_{\mathbf{K}\mathrm{\mathrm{\Sigma}}}(\emptyset, \bigoplus C_{h+1, n}\rightarrow\bigvee C_{h+1, n})$, $\mathrm{Pbl}_{\mathbf{K}\mathrm{\mathrm{\Sigma}}}(\emptyset, \neg\varphi)$ holds. It is a contradiction that $\varphi$ is $\mathbf{K}\mathrm{\mathrm{\Sigma}}$-consistent. Thus there exists $\alpha\in C_{h+1. n}$ such that $\mathrm{Pbl}_{\mathbf{K}\mathrm{\mathrm{\Sigma}}}(\emptyset, \alpha\rightarrow\varphi)$ holds, so $S$ is not empty. Since $\varphi$ is $\mathbf{K}\mathrm{\mathrm{\Sigma}}$-consistent, there exists $\alpha\in S$ such that $\alpha$ is $\mathbf{K}\mathrm{\mathrm{\Sigma}}$-consistent.
\end{proof}

\begin{lem}[Existence lemma, $\mathrm{RCA}_0$]\label{exlem}
Let $\varphi\in\mathcal{L}_{h, n}$ and $\psi\in\mathrm{Fml}$. If $\psi\land\Diamond\varphi$ is $\mathbf{K}\mathrm{\mathrm{\Sigma}}$-consistent, then there exists $\alpha\in C_{h, n}$ such that $\psi\land\Diamond\alpha$ is $\mathbf{K}\mathrm{\mathrm{\Sigma}}$-consistent and $\mathrm{Pbl}_{\mathbf{K}\mathrm{\mathrm{\Sigma}}}(\emptyset, \alpha\rightarrow\varphi)$ holds.
\end{lem}

\begin{proof}
Let $S=\{\alpha\in C_{h, n}\mid\mathrm{Pbl}_{\mathbf{K}\mathrm{\mathrm{\Sigma}}}(\emptyset, \alpha\rightarrow\varphi)\}$. Then, $\mathrm{Pbl}_{\mathbf{K}\mathrm{\mathrm{\Sigma}}}(\emptyset, \Diamond\varphi\leftrightarrow\Diamond(\bigvee S)\leftrightarrow(\bigvee\{\Diamond\alpha\mid\alpha\in S\}))$ holds. Thus,
\begin{equation*}
\mathrm{Pbl}_{\mathbf{K}\mathrm{\mathrm{\Sigma}}}(\emptyset, (\psi\land\Diamond\varphi)\leftrightarrow(\psi\land\bigvee\{\Diamond\alpha\mid\alpha\in S\})\leftrightarrow(\bigvee\{\psi\land\Diamond\alpha\mid\alpha\in S\})).
\end{equation*}

Since $\psi\land\Diamond\varphi$ is $\mathbf{K}\mathrm{\mathrm{\Sigma}}$-consistent, there exists $\alpha\in S$ such that $\psi\land\Diamond\alpha$ is $\mathbf{K}\mathrm{\mathrm{\Sigma}}$-consistent.
\end{proof}

With the above preparation, we first show the weak completeness theorem for the case of $4\not\in\mathrm{\Sigma}$ in $\mathrm{RCA}_0$.

\begin{dfn}[$\mathrm{RCA}_0$]\label{wn4first}
We assume that $\mathrm{\Sigma}\subseteq\{T, B, D\}$. Let $\varphi\in\mathcal{L}_{h, n}$ be $\mathbf{K}\mathrm{\mathrm{\Sigma}}$-consistent. Then we define $C=(\tilde{W}, \xrightarrow{c}, \tilde{V})$ satisfying the followings:
\begin{itemize}
    \item $\tilde{W}=\{\alpha\in C_{h+1, n}\mid\neg\mathrm{Pbl}_{\mathbf{K}\mathrm{\mathrm{\Sigma}}}(\alpha, \bot)\land\mathrm{Pbl}_{\mathbf{K}\mathrm{\mathrm{\Sigma}}}(\emptyset, \alpha\rightarrow\varphi)\}$,
    \item $\xrightarrow{c}=\{(\alpha, \beta)\in\tilde{W}\times\tilde{W}\mid\neg\mathrm{Pbl}_{\mathbf{K}\mathrm{\mathrm{\Sigma}}}(\alpha\land\Diamond\beta, \bot)\}$,
    \item For all $\psi\in sub(\varphi)$ and $\alpha\in\tilde{W}$, $\tilde{V}(\alpha, \psi)=1\iff\mathrm{Pbl}_{\mathbf{K}\mathrm{\mathrm{\Sigma}}}(\emptyset, \alpha\rightarrow\psi)$.
\end{itemize}
\end{dfn}

This $C$ exists by Lemma~\ref{select} and $\mathrm{I}\mathrm{\Sigma}^0_1$ (more precisely, we need the bounded $\mathrm{\Sigma}^0_1$-comprehension to obtain $\xrightarrow{c}$).
Moreover it is clear that $\tilde{V}(\alpha, \varphi)=1$ for any $\alpha\in\tilde{W}$.  By Lemma~\ref{notempty} and Lemma~\ref{exlem}, $\tilde{W}$ and $\xrightarrow{c}$ are non-empty sets.

\begin{lem}[$\mathrm{RCA}_0$]\label{Bsym}
The followings hold:
\begin{enumerate}
    \item if $T\in\mathrm{\Sigma}$ then $\xrightarrow{c}$ is reflexive.
    \item if $B\in\mathrm{\Sigma}$ then $\xrightarrow{c}$ is symmetric.
    \item if $D\in\mathrm{\Sigma}$ then $\xrightarrow{c}$ is serial.
    \item $C$ is a weak Kripke model.
\end{enumerate}
\end{lem}
\begin{proof}
(1): Let $\alpha\in\tilde{W}$. Assume that $\alpha\land\Diamond\alpha$ is $\mathbf{K}\mathrm{\mathrm{\Sigma}}$-inconsistent. Then $\mathrm{Pbl}_{\mathbf{K}\mathrm{\mathrm{\Sigma}}}(\emptyset, \alpha\rightarrow\Box\neg\alpha)$ holds. By $T\in\mathrm{\Sigma}$, $\mathrm{Pbl}_{\mathbf{K}\mathrm{\mathrm{\Sigma}}}(\emptyset, \Box\neg\alpha\rightarrow\neg\alpha)$ holds, thus $\mathrm{Pbl}_{\mathbf{K}\mathrm{\mathrm{\Sigma}}}(\emptyset, \alpha\rightarrow\neg\alpha)$ holds. It is a contradiction that $\alpha$ is $\mathbf{K}\mathrm{\mathrm{\Sigma}}$-consistent.

(2): Let $\alpha, \beta\in\tilde{W}$ such that $\alpha\land\Diamond\beta$ is $\mathbf{K}\mathrm{\mathrm{\Sigma}}$-consistent. Assume that $\beta\land\Diamond\alpha$ is $\mathbf{K}\mathrm{\mathrm{\Sigma}}$-inconsistent. Then $\mathrm{Pbl}_{\mathbf{K}\mathrm{\mathrm{\Sigma}}}(\emptyset, \beta\rightarrow\Box\neg\alpha)$ holds, so $\mathrm{Pbl}_{\mathbf{K}\mathrm{\mathrm{\Sigma}}}(\emptyset, \Diamond\beta\rightarrow\neg\Box\Diamond\alpha)$ holds. On the other hand, by $B\in\mathrm{\Sigma}$, $\mathrm{Pbl}_{\mathbf{K}\mathrm{\mathrm{\Sigma}}}(\emptyset, \alpha\rightarrow\Box\Diamond\alpha)$ holds. It is a contradiction that $\alpha\land\Diamond\beta$ is $\mathbf{K}\mathrm{\mathrm{\Sigma}}$-consistent.

(3): Let $\alpha\in\tilde{W}$. We show that there exists $\beta\in\tilde{W}$ such that $\neg\mathrm{Pbl}_{\mathbf{K}\mathrm{\mathrm{\Sigma}}}(\alpha\land\Diamond\beta, \bot)$ holds. First, we show that $\alpha\land\Diamond\varphi$ is $\mathbf{K}\mathrm{\mathrm{\Sigma}}$-consistent. Assume that $\alpha\land\Diamond\varphi$ is $\mathbf{K}\mathrm{\mathrm{\Sigma}}$-inconsistent. Then $\mathrm{Pbl}_{\mathbf{K}\mathrm{\mathrm{\Sigma}}}(\emptyset, \alpha\rightarrow\Box\neg\varphi)$ holds. By $D\in\mathrm{\Sigma}$, $\mathrm{Pbl}_{\mathbf{K}\mathrm{\mathrm{\Sigma}}}(\emptyset, \Box\neg\varphi\rightarrow\Diamond\neg\varphi)$ holds, so $\mathrm{Pbl}_{\mathbf{K}\mathrm{\mathrm{\Sigma}}}(\emptyset, \alpha\rightarrow\Diamond\neg\varphi)$ holds. Since $\mathrm{Pbl}_{\mathbf{K}\mathrm{\mathrm{\Sigma}}}(\Diamond\neg\varphi, \neg\varphi)$ holds, thus $\mathrm{Pbl}_{\mathbf{K}\mathrm{\mathrm{\Sigma}}}(\alpha, \neg\varphi)$ holds. It is a contradiction that $\alpha$ is $\mathbf{K}\mathrm{\mathrm{\Sigma}}$-consistent. Thus by Lemma~\ref{exlem}, there exists $\beta\in C_{h+1, n}$ such that $\alpha\land\Diamond\beta$ is $\mathbf{K}\mathrm{\mathrm{\Sigma}}$-consistent and $\mathrm{Pbl}_{\mathbf{K}\mathrm{\mathrm{\Sigma}}}(\emptyset, \beta\rightarrow\varphi)$ holds. Thus it is enough to show that $\beta$ is $\mathbf{K}\mathrm{\mathrm{\Sigma}}$-consistent. Assume that $\mathrm{Pbl}_{\mathbf{K}\mathrm{\mathrm{\Sigma}}}(\emptyset, \beta\rightarrow\bot)$ holds. Then since $\mathrm{Pbl}_{\mathbf{K}\mathrm{\mathrm{\Sigma}}}(\emptyset, \Diamond\bot\leftrightarrow\bot)$ and $\mathrm{Pbl}_{\mathbf{K}\mathrm{\mathrm{\Sigma}}}(\emptyset, \Diamond\beta\rightarrow\Diamond\bot)$ hold, so $\mathrm{Pbl}_{\mathbf{K}\mathrm{\mathrm{\Sigma}}}(\emptyset, \Diamond\beta\rightarrow\bot)$ holds. Thus $\mathrm{Pbl}_{\mathbf{K}\mathrm{\mathrm{\Sigma}}}(\emptyset, \alpha\land\Diamond\beta\rightarrow\bot)$ holds. It is a contradiction that $\alpha\land\Diamond\beta$ is $\mathbf{K}\mathrm{\mathrm{\Sigma}}$-consistent.

(4): We show that for all $\alpha\in\tilde{W}$, $\tilde{V}(\alpha, \psi)$ satisfies the conditions of a valuation by induction on the length of $\psi$, so we can prove it in $\mathrm{I}\mathrm{\Sigma}^0_1$ (Note that $C$ is finite). We show it in the case of $\psi\equiv\Box\theta$, that is, for all $\alpha\in\tilde{W}$,
\begin{equation*}
\tilde{V}(\alpha, \Box\theta)=1\iff\forall \beta\in \tilde{W}(\alpha\xrightarrow{c}\beta\Rightarrow\tilde{V}(\beta, \theta)=1).
\end{equation*}

Fix $\alpha\in\tilde{W}$. We have two cases. First, we assume that $\tilde{V}(\alpha, \Box\theta)=0$. Then, by Lemma~\ref{select}, $\mathrm{Pbl}_{\mathbf{K}\mathrm{\mathrm{\Sigma}}}(\emptyset, \alpha\rightarrow\neg\Box\theta)$ holds. Thus $\alpha\land\Diamond\neg\theta$ is $\mathbf{K}\mathrm{\mathrm{\Sigma}}$-consistent, so by Lemma~\ref{exlem}, there exists $\beta$ such that $\mathrm{Pbl}_{\mathbf{K}\mathrm{\mathrm{\Sigma}}}(\emptyset, \beta\rightarrow\neg\theta)$ holds. Moreover, $\beta$ is $\mathbf{K}\mathrm{\mathrm{\Sigma}}$-consistent since $\alpha\land\Diamond\beta$ is $\mathbf{K}\mathrm{\mathrm{\Sigma}}$-consistent. Thus $\neg\mathrm{Pbl}_{\mathbf{K}\mathrm{\mathrm{\Sigma}}}(\emptyset, \beta\rightarrow\theta)$ holds, so $\tilde{V}(\beta, \theta)=0$. It is enough to show that $\mathrm{Pbl}_{\mathbf{K}\mathrm{\mathrm{\Sigma}}}(\emptyset, \beta\rightarrow\varphi)$ holds. Assume that $\neg\mathrm{Pbl}_{\mathbf{K}\mathrm{\mathrm{\Sigma}}}(\emptyset, \beta\rightarrow\varphi)$ holds. Then, by Lemma~\ref{select}, $\mathrm{Pbl}_{\mathbf{K}\mathrm{\mathrm{\Sigma}}}(\emptyset, \beta\rightarrow\neg\varphi)$ holds, so $\mathrm{Pbl}_{\mathbf{K}\mathrm{\mathrm{\Sigma}}}(\emptyset, \Diamond\beta\rightarrow\neg\Box\varphi)$ holds. On the other hand, $\mathrm{Pbl}_{\mathbf{K}\mathrm{\mathrm{\Sigma}}}(\alpha, \Box\varphi)$ holds. It is a contradiction that $\alpha\land\Diamond\beta$ is $\mathbf{K}\mathrm{\mathrm{\Sigma}}$-consistent. 

The other case is when $\tilde{V}(\alpha, \Box\theta)=1$, that is, $\mathrm{Pbl}_{\mathbf{K}\mathrm{\mathrm{\Sigma}}}(\emptyset, \alpha\rightarrow\Box\theta)$ holds.  Fix $\beta\in\tilde{W}$ such that $\alpha\xrightarrow{c}\beta$. Assume that $\neg\mathrm{Pbl}_{\mathbf{K}\mathrm{\mathrm{\Sigma}}}(\emptyset, \beta\rightarrow\theta)$ holds. Then, by Lemma~\ref{select}, $\mathrm{Pbl}_{\mathbf{K}\mathrm{\mathrm{\Sigma}}}(\emptyset, \beta\rightarrow\neg\theta)$ holds, so $\mathrm{Pbl}_{\mathbf{K}\mathrm{\mathrm{\Sigma}}}(\emptyset, \Diamond\beta\rightarrow\neg\Box\theta)$ holds. It is a contradiction that $\alpha\land\Diamond\beta$ is $\mathbf{K}\mathrm{\mathrm{\Sigma}}$-consistent. Thus $\mathrm{Pbl}_{\mathbf{K}\mathrm{\mathrm{\Sigma}}}(\emptyset, \beta\rightarrow\theta)$ holds, that is, $\tilde{V}(\beta, \theta)=1$.
\end{proof}

\begin{prop}[$\mathrm{RCA}_0$]\label{wn4}
Let $\mathrm{\Sigma}\subseteq\{T, B, D\}$. If $\varphi\in\mathrm{Fml}$ is $\mathbf{K}\mathrm{\mathrm{\Sigma}}$-consistent, then there exist a frame which is appropriate to $\mathrm{\Sigma}$ and a valuation $V$ such that $(F, V)$ is a weak model of $\varphi$.
\end{prop}
\begin{proof}
By Definition~\ref{wn4first}, we define $F=(\tilde{W}, \xrightarrow{c})$ and $V=\tilde{V}$. By Lemma~\ref{Bsym}, $(F, V)$ is a weak model of $\varphi$.
\end{proof}

We next aim to construct a weak Kripke model for the case $4\in\mathrm{\Sigma}$.
In what follows, we will assume $4\in\mathrm{\Sigma}$ unless otherwise specified. 

\begin{lem}[$\mathrm{RCA}_0$]\label{ea}
For all $\mathbf{K}\mathrm{\mathrm{\Sigma}}$-consistent $\alpha\in C_{h+1, n}$, there is only one $\mathbf{K}\mathrm{\mathrm{\Sigma}}$-consistent formula $\alpha^{\prime}\in C_{h, n}$ such that $\mathrm{Pbl}_{\mathbf{K}\mathrm{\mathrm{\Sigma}}}(\emptyset, \alpha\rightarrow\alpha^{\prime})$ holds.
\end{lem}
\begin{proof}
Let $\mathrm{\Sigma_0}=\mathrm{\Sigma}\setminus\{4\}$ and let $\alpha\in C_{h+1, n}$ be $\mathbf{K}\mathrm{\mathrm{\Sigma}}$-consistent. Then $\alpha$ is $\mathbf{K}\mathrm{\mathrm{\Sigma}}_0$-consistent. Note that $\bigoplus C_{h, n}\in\mathcal{L}_{h, n}$. By Proposition~\ref{bigo}, Lemma~\ref{select}, and $\alpha$ is $\mathbf{K}\mathrm{\mathrm{\Sigma}}_0$-consistent, $\mathrm{Pbl}_{\mathbf{K}\mathrm{\mathrm{\Sigma}}_0}(\emptyset, \alpha\rightarrow\bigoplus C_{h, n})$ holds. Thus Proposition~\ref{wn4}, $\bigoplus C_{h, n}$ has a weak model $(W, \xrightarrow{c}, V)$. Therefore \\$V(\alpha, \bigoplus C_{h, n})=1$. By the definition of $\bigoplus$, there is only one $\alpha^{\prime}\in C_{h, n}$ such that $V(\alpha, \alpha^{\prime})=1$. Thus $\mathrm{Pbl}_{\mathbf{K}\mathrm{\mathrm{\Sigma}}_0}(\emptyset, \alpha\rightarrow\alpha^{\prime})$ holds, so $\mathrm{Pbl}_{\mathbf{K}\mathrm{\mathrm{\Sigma}}}(\emptyset, \alpha\rightarrow\alpha^{\prime})$ holds. Since $\alpha$ is $\mathbf{K}\mathrm{\mathrm{\Sigma}}$-consistent, $\alpha^{\prime}$ is $\mathbf{K}\mathrm{\mathrm{\Sigma}}$-consistent.
\end{proof}

Note that, for each $n\in\mathbb{N}$, we may consider the correspondence $\alpha\mapsto \alpha'$ in Lemma~\ref{ea} as a function (with the finite domain) by $\mathrm{I}\mathrm{\Sigma}^0_1$.
Next we consider the following two models according to the conditions of $\mathrm{\Sigma}$.

\begin{dfn}[$\mathrm{RCA}_0$]\label{m1}
We assume that $4\in\mathrm{\Sigma}$ and $B\not\in\mathrm{\Sigma}$. Let $\varphi\in\mathcal{L}_{h, n}$ be $\mathbf{K}\mathrm{\mathrm{\Sigma}}$-consistent. Then we define $N=(\tilde{W}, \xrightarrow{m}, \tilde{V})$ satisfying the followings:
\begin{itemize}
    \item $\tilde{W}=\{\alpha\in C_{h+1, n}\mid\neg\mathrm{Pbl}_{\mathbf{K}\mathrm{\mathrm{\Sigma}}}(\alpha,\bot)\land\mathrm{Pbl}_{\mathbf{K}\mathrm{\mathrm{\Sigma}}}(\emptyset, \alpha\rightarrow\varphi)\}$,
    \item $\xrightarrow{m}=\{(\alpha, \beta)\in\tilde{W}\times\tilde{W}\mid\mathrm{Pbl}_{\mathbf{K}\mathrm{\mathrm{\Sigma}}}(\emptyset, \alpha\rightarrow\Diamond\beta^{\prime})\land\forall\xi\in C_{h, n}(\mathrm{Pbl}_{\mathbf{K}\mathrm{\mathrm{\Sigma}}}(\emptyset, \beta\rightarrow\Diamond\xi)\rightarrow\mathrm{Pbl}_{\mathbf{K}\mathrm{\mathrm{\Sigma}}}(\emptyset, \alpha\rightarrow\Diamond\xi))\}$,
    \item For all $\psi\in sub(\varphi)$ and $\alpha\in\tilde{W}$, $\tilde{V}(\alpha, \psi)=1\iff\mathrm{Pbl}_{\mathbf{K}\mathrm{\mathrm{\Sigma}}}(\emptyset, \alpha\rightarrow\psi)$.
\end{itemize}
\end{dfn}

This $N$ exists by Lemma~\ref{select} and $\mathrm{I}\mathrm{\Sigma}^0_1$. 

\begin{dfn}[$\mathrm{RCA}_0$]\label{m2}
We assume that $4, B\in\mathrm{\Sigma}$. Let $\varphi\in\mathcal{L}_{h, n}$ be $\mathbf{K}\mathrm{\mathrm{\Sigma}}$-consistent. Then we define $M=(\tilde{W}, \xrightarrow{mm}, \tilde{V})$ satisfying the followings:
\begin{itemize}
    \item $\tilde{W}=\{\alpha\in C_{h+1, n}\mid\neg\mathrm{Pbl}_{\mathbf{K}\mathrm{\mathrm{\Sigma}}}(\alpha,\bot)\land\mathrm{Pbl}_{\mathbf{K}\mathrm{\mathrm{\Sigma}}}(\emptyset, \alpha\rightarrow\varphi)\}$,
    \item $\xrightarrow{mm}=\{(\alpha, \beta)\in\tilde{W}\times\tilde{W}\mid\alpha\xrightarrow{m}\beta\land\beta\xrightarrow{m}\alpha\}$,
    \item For all $\psi\in sub(\varphi)$ and $\alpha\in\tilde{W}$, $\tilde{V}(\alpha, \psi)=1\iff\mathrm{Pbl}_{\mathbf{K}\mathrm{\mathrm{\Sigma}}}(\emptyset, \alpha\rightarrow\psi)$.
\end{itemize}
\end{dfn}

This $M$ exists by Lemma~\ref{select} and $\mathrm{I}\mathrm{\Sigma}^0_1$. 

\begin{prop}[$\mathrm{RCA}_0$]\label{crm}
Let $\alpha, \beta\in\tilde{W}$ and $4\in\mathrm{\Sigma}$. Then $\alpha\xrightarrow{c}\beta$ implies $\alpha\xrightarrow{m}\beta$. Moreover, if $4, B\in\mathrm{\Sigma}$, then $\alpha\xrightarrow{c}\beta$ implies $\alpha\xrightarrow{mm}\beta$. 
\end{prop}
\begin{proof}
Suppose $\alpha\xrightarrow{c}\beta$. By $\mathrm{Pbl}_{\mathbf{K}\mathrm{\mathrm{\Sigma}}}(\emptyset, \beta\rightarrow\beta^{\prime})$, $\mathrm{Pbl}_{\mathbf{K}\mathrm{\mathrm{\Sigma}}}(\emptyset, \Diamond\beta\rightarrow\Diamond\beta^{\prime})$ holds. Since $\alpha\land\Diamond\beta$ is $\mathbf{K}\mathrm{\mathrm{\Sigma}}$-consistent, $\mathrm{Pbl}_{\mathbf{K}\mathrm{\mathrm{\Sigma}}}(\emptyset, \alpha\rightarrow\Diamond\beta^{\prime})$ holds. Let $\gamma\in C_{h, n}$. Assume that $\mathrm{Pbl}_{\mathbf{K}\mathrm{\mathrm{\Sigma}}}(\emptyset, \beta\rightarrow\Diamond\gamma)$ and $\neg\mathrm{Pbl}_{\mathbf{K}\mathrm{\mathrm{\Sigma}}}(\emptyset, \alpha\rightarrow\Diamond\gamma)$ holds. Then $\mathrm{Pbl}_{\mathbf{K}\mathrm{\mathrm{\Sigma}}}(\emptyset, \Diamond\beta\rightarrow\Diamond\Diamond\gamma)$ holds. By $4\in\mathrm{\Sigma}$, $\mathrm{Pbl}_{\mathbf{K}\mathrm{\mathrm{\Sigma}}}(\emptyset, \Diamond\Diamond\gamma\rightarrow\Diamond\gamma)$, so $\mathrm{Pbl}_{\mathbf{K}\mathrm{\mathrm{\Sigma}}}(\emptyset, \Diamond\beta\rightarrow\Diamond\gamma)$ holds. On the other hands, by Lemma~\ref{select}, $\mathrm{Pbl}_{\mathbf{K}\mathrm{\mathrm{\Sigma}}}(\emptyset, \alpha\rightarrow\neg\Diamond\gamma)$. This is the contradiction that $\alpha\land\Diamond\beta$ is $\mathbf{K}\mathrm{\mathrm{\Sigma}}$-consistent. Similarly, it also holds in the case of $4, B\in\mathrm{\Sigma}$ by Lemma~\ref{Bsym}.
\end{proof}

\begin{lem}[$\mathrm{RCA}_0$]\label{trs}
The followings hold:
\begin{enumerate}
    \item $\xrightarrow{m}$ is transitive.
    \item if $T\in\mathrm{\Sigma}$ then $\xrightarrow{m}$ is reflexive.
    \item if $D\in\mathrm{\Sigma}$ then $\xrightarrow{m}$ is serial.
\end{enumerate}
\end{lem}
\begin{proof}
(1): Let $\alpha\xrightarrow{m}\beta$ and $\beta\xrightarrow{m}\gamma$. By $\mathrm{Pbl}_{\mathbf{K}\mathrm{\mathrm{\Sigma}}}(\emptyset, \beta\rightarrow\Diamond\gamma^{\prime})$ and $\alpha\xrightarrow{m}\beta$, $\mathrm{Pbl}_{\mathbf{K}\mathrm{\mathrm{\Sigma}}}(\emptyset, \alpha\rightarrow\Diamond\gamma^{\prime})$ holds. Let $\delta\in C_{h, n}$ such that $\mathrm{Pbl}_{\mathbf{K}\mathrm{\mathrm{\Sigma}}}(\emptyset, \gamma\rightarrow\Diamond\delta)$ holds. Then, by $\beta\xrightarrow{m}\gamma$, $\mathrm{Pbl}_{\mathbf{K}\mathrm{\mathrm{\Sigma}}}(\emptyset, \beta\rightarrow\Diamond\delta)$ holds. By $\alpha\xrightarrow{m}\beta$, $\mathrm{Pbl}_{\mathbf{K}\mathrm{\mathrm{\Sigma}}}(\emptyset, \alpha\rightarrow\Diamond\delta)$ holds.

(2): By $T\in\mathrm{\Sigma}$, $\mathrm{Pbl}_{\mathbf{K}\mathrm{\mathrm{\Sigma}}}(\emptyset, \alpha\rightarrow\Diamond\alpha)$ holds, and $\mathrm{Pbl}_{\mathbf{K}\mathrm{\mathrm{\Sigma}}}(\emptyset, \alpha\rightarrow\alpha^{\prime})$ holds, so $\mathrm{Pbl}_{\mathbf{K}\mathrm{\mathrm{\Sigma}}}(\emptyset, \Diamond\alpha\rightarrow\Diamond\alpha^{\prime})$ holds. Thus $\mathrm{Pbl}_{\mathbf{K}\mathrm{\mathrm{\Sigma}}}(\emptyset, \alpha\rightarrow\Diamond\alpha^{\prime})$ holds.

(3): Let $\alpha\in\tilde{W}$. First, we show that $\mathrm{Pbl}_{\mathbf{K}\mathrm{\mathrm{\Sigma}}}(\emptyset, \alpha\rightarrow\Diamond\varphi)$ holds. Assume that $\neg\mathrm{Pbl}_{\mathbf{K}\mathrm{\mathrm{\Sigma}}}(\emptyset, \alpha\rightarrow\Diamond\varphi)$ holds. Then $\mathrm{Pbl}_{\mathbf{K}\mathrm{\mathrm{\Sigma}}}(\emptyset, \alpha\rightarrow\Box\neg\varphi)$ holds. By $D\in\mathrm{\Sigma}$, $\mathrm{Pbl}_{\mathbf{K}\mathrm{\mathrm{\Sigma}}}(\emptyset, \Box\neg\varphi\rightarrow\Diamond\neg\varphi)$ holds, so $\mathrm{Pbl}_{\mathbf{K}\mathrm{\mathrm{\Sigma}}}(\emptyset, \alpha\rightarrow\Diamond\neg\varphi)$ holds. \\Since $\mathrm{Pbl}_{\mathbf{K}\mathrm{\mathrm{\Sigma}}}(\Diamond\neg\varphi, \neg\varphi)$ holds, thus $\mathrm{Pbl}_{\mathbf{K}\mathrm{\mathrm{\Sigma}}}(\alpha, \neg\varphi)$ holds. It is a contradiction that $\alpha$ is $\mathbf{K}\mathrm{\mathrm{\Sigma}}$-consistent. Thus by Lemma~\ref{exlem}, there exists $\beta\in C_{h+1, n}$ such that $\alpha\land\Diamond\beta$ is $\mathbf{K}\mathrm{\mathrm{\Sigma}}$-consistent and $\mathrm{Pbl}_{\mathbf{K}\mathrm{\mathrm{\Sigma}}}(\emptyset, \beta\rightarrow\varphi)$ holds. Then $\beta$ is $\mathbf{K}\mathrm{\mathrm{\Sigma}}$-consistent. So $\alpha\xrightarrow{c}\beta$. By Proposition~\ref{crm}, $\alpha\xrightarrow{m}\beta$.
\end{proof}

\begin{lem}[$\mathrm{RCA}_0$]\label{mmsuper}
The followings hold:
\begin{enumerate}
    \item $\xrightarrow{mm}$ is symmetric and transitive.
    \item if $T\in\mathrm{\Sigma}$ then $\xrightarrow{mm}$ is reflexive.
    \item if $D\in\mathrm{\Sigma}$ then $\xrightarrow{mm}$ is serial.
\end{enumerate}
\end{lem}
\begin{proof}
By Lemma~\ref{trs}.
\end{proof}

\begin{lem}[Generalized truth lemma, $\mathrm{RCA}_0$]\label{generalized}
Let $\varphi\in\mathcal{L}_{h, n}$ be $\mathbf{K}\mathrm{\mathrm{\Sigma}}$-consistent. Assume that $R\subseteq\tilde{W}\times\tilde{W}$ satisfies that $\alpha\xrightarrow{c}\beta$ implies $\alpha R\beta$ and $\alpha R\beta$ implies $\mathrm{Pbl}_{\mathbf{K}\mathrm{\mathrm{\Sigma}}}(\emptyset, \alpha\rightarrow\Diamond\beta^{\prime})$. Then $(\tilde{W}, R, \tilde{V})$ is a weak Kripke model.
\end{lem}
\begin{proof}
We show that for all $\alpha\in\tilde{W}$, $\tilde{V}(\alpha, \psi)$ satisfies the conditions of a valuation by induction on the length of $\psi$, so we can prove it in $\mathrm{I}\mathrm{\Sigma}^0_1$ (Note that $(\tilde{W}, R, \tilde{V})$ is finite). We show it in the case of $\psi\equiv\Box\theta$, that is, for all $\alpha\in\tilde{W}$,
\begin{equation*}
\tilde{V}(\alpha, \Box\theta)=1\iff\forall \beta\in \tilde{W}(\alpha R\beta\Rightarrow\tilde{V}(\beta, \theta)=1).
\end{equation*}

Fix $\alpha\in\tilde{W}$. We have two cases. First, we assume that $\tilde{V}(\alpha, \Box\theta)=0$. Then by Lemma~\ref{select}, $\mathrm{Pbl}_{\mathbf{K}\mathrm{\mathrm{\Sigma}}}(\emptyset, \alpha\rightarrow\neg\Box\theta)$ holds. Thus $\alpha\land\Diamond\neg\theta$ is $\mathbf{K}\mathrm{\mathrm{\Sigma}}$-consistent, so by Lemma~\ref{exlem}, there exists $\beta$ such that $\mathrm{Pbl}_{\mathbf{K}\mathrm{\mathrm{\Sigma}}}(\emptyset, \beta\rightarrow\neg\theta)$ holds. Moreover, $\beta$ is $\mathbf{K}\mathrm{\mathrm{\Sigma}}$-consistent since $\alpha\land\Diamond\beta$ is $\mathbf{K}\mathrm{\mathrm{\Sigma}}$-consistent. Then $\alpha\xrightarrow{c}\beta$, thus $\alpha R\beta$. Thus $\neg\mathrm{Pbl}_{\mathbf{K}\mathrm{\mathrm{\Sigma}}}(\emptyset, \beta\rightarrow\theta)$ holds, so $\tilde{V}(\beta, \theta)=0$. 

The other case is when $\tilde{V}(\alpha, \Box\theta)=1$, that is, $\mathrm{Pbl}_{\mathbf{K}\mathrm{\mathrm{\Sigma}}}(\emptyset, \alpha\rightarrow\Box\theta)$ holds.  Fix $\beta\in\tilde{W}$ such that $\alpha R\beta$. $\mathrm{Pbl}_{\mathbf{K}\mathrm{\mathrm{\Sigma}}}(\emptyset, \alpha\rightarrow\Diamond\beta^{\prime})$ holds. Assume that $\mathrm{Pbl}_{\mathbf{K}\mathrm{\mathrm{\Sigma}}}(\emptyset, \beta^{\prime}\rightarrow\neg\theta)$ holds. Then $\mathrm{Pbl}_{\mathbf{K}\mathrm{\mathrm{\Sigma}}}(\emptyset, \Diamond\beta^{\prime}\rightarrow\neg\Box\theta)$ holds, thus $\mathrm{Pbl}_{\mathbf{K}\mathrm{\mathrm{\Sigma}}}(\emptyset, \alpha\rightarrow\neg\Box\theta)$ holds. It is the contradiction that $\alpha$ is $\mathbf{K}\mathrm{\mathrm{\Sigma}}$-consistent. Thus $\mathrm{Pbl}_{\mathbf{K}\mathrm{\mathrm{\Sigma}}}(\emptyset, \beta^{\prime}\rightarrow\theta)$ holds. Then $\mathrm{Pbl}_{\mathbf{K}\mathrm{\mathrm{\Sigma}}}(\emptyset, \beta\rightarrow\beta^{\prime})$ holds, so $\mathrm{Pbl}_{\mathbf{K}\mathrm{\mathrm{\Sigma}}}(\emptyset, \beta\rightarrow\theta)$ holds. 
\end{proof}

From the above, we now see the weak completeness theorem for the case of $4\in\mathrm{\Sigma}$ in $\mathrm{RCA}_0$.

\begin{prop}[$\mathrm{RCA}_0$]
Let $\mathrm{\Sigma}\subseteq\{T, 4, B, D\}$ and $4\in\mathrm{\Sigma}$. If $\varphi\in\mathrm{Fml}$ is $\mathbf{K}\mathrm{\mathrm{\Sigma}}$-consistent, then there exist a frame which is appropriate to $\mathrm{\Sigma}$ and a valuation $V$ such that $(F, V)$ is a weak model of $\varphi$.
\end{prop}

\begin{proof}
Assume that $\varphi\in\mathcal{L}_{h, n}$. By Definition~\ref{m1} and Definition~\ref{m2}, we define
\begin{itemize}
    \item Case1: If $B\not\in\mathrm{\Sigma}$, then we define $F=(\tilde{W}, \xrightarrow{m})$ and $V=\tilde{V}$.
    \item Case2: If $B\in\mathrm{\Sigma}$, then we define $F=(\tilde{W}, \xrightarrow{mm})$ and $V=\tilde{V}$.
\end{itemize}

By Generalized truth lemma, $(F, V)$ is a weak model of $\varphi$.
\end{proof}

Summarizing, we have the following weak completeness theorem with the finite model property. 
\begin{thm}[Weak completeness theorem, $\mathrm{RCA}_0$]
Let $\mathrm{\Sigma}\subseteq\{T, 4, B, D\}$. If $\varphi\in\mathrm{Fml}$ is $\mathbf{K}\mathrm{\mathrm{\Sigma}}$-consistent, then there exist a finite frame $F$ which is appropriate to $\mathrm{\Sigma}$ and a valuation $V$ such that $(F, V)$ is a weak model of $\varphi$.
\end{thm}

\section{Weak Completeness Theorem for \textbf{GL}}
In this section, we will prove the weak completeness theorem for $\mathbf{GL}$ in $\mathrm{RCA}_0$. 
\begin{prop}[$\mathrm{RCA}_0$]\label{fin}
Let $F=(W, R)$ be a finite transitive frame such that $R$ is irreflexive i.e., $\forall w\in W\neg(wRw)$. Then $F$ is appropriate to $L$.
\end{prop}

\begin{proof}
Suppose $|W|=m$. Assume that $F$ is not appropriate to $L$. Then $W$ has an infinite assending sequence by $R$, so $(\exists f:\mathbb{N}\rightarrow W)(\forall n\in\mathbb{N})(f(n)Rf(n+1))$ holds. We consider $f\upharpoonright m+1:m+1\rightarrow W$. By $\mathrm{I}\mathrm{\Sigma}^0_1$, we can show that there exist $i<j<m+1$ such that $f(i)=f(j)$. Let $w_0=f(i)$. Since $R$ is transitive, $f(i)Rf(j)$. Thus $w_0Rw_0$. It is a contradiction that $R$ is irreflexive.
\end{proof}

\begin{dfn}[$\mathrm{RCA}_0$]
Let $\varphi\in\mathcal{L}_{h, n}$ be $\mathbf{GL}$-consistent. Then we define $D=(\tilde{W}, \xrightarrow{d}, \tilde{V})$ satisfying the followings:
\begin{itemize}
    \item $\tilde{W}=\{\alpha\in C_{h+1, n}\mid\neg\mathrm{Pbl}_{\mathbf{GL}}(\alpha,\bot)\land\mathrm{Pbl}_{\mathbf{GL}}(\emptyset, \alpha\rightarrow\varphi)\}$,
    \item $\xrightarrow{d}=\{(\alpha, \beta)\in\tilde{W}\times\tilde{W}\mid\alpha\xrightarrow{m}\beta\land\exists\gamma\in C_{h, n}(\mathrm{Pbl}_{\mathbf{GL}}(\emptyset, \alpha\rightarrow\Diamond\gamma)\land\mathrm{Pbl}_{\mathbf{GL}}(\emptyset, \beta\rightarrow\Box\neg\gamma))\}$,
    \item For all $\psi\in sub(\varphi)$ and $\alpha\in\tilde{W}$, $\tilde{V}(\alpha, \psi)=1\iff\mathrm{Pbl}_{\mathbf{GL}}(\emptyset, \alpha\rightarrow\psi)$.
\end{itemize}
\end{dfn}

This $D$ exists by Lemma~\ref{select} and $\mathrm{I}\mathrm{\Sigma}^0_1$. 

\begin{lem}[$\mathrm{RCA}_0$]
The followings hold:
\begin{enumerate}
    \item $\xrightarrow{d}$ is transitive.
    \item $\tilde{W}$ has no infinite assending sequences by $\xrightarrow{d}$.
\end{enumerate}
\end{lem}

\begin{proof}
(1): Let $\alpha\xrightarrow{d}\beta$ and $\beta\xrightarrow{d}\gamma$. By Lemma~\ref{trs}, $\alpha\xrightarrow{m}\gamma$. Then there exist $\delta_0, \delta_1\in C_{h, n}$ such that $\mathrm{Pbl}_{\textbf{GL}}(\emptyset, \alpha\rightarrow\Diamond\delta_0)$ and $\mathrm{Pbl}_{\textbf{GL}}(\emptyset, \beta\rightarrow\Box\neg\delta_0)$ and $\mathrm{Pbl}_{\textbf{GL}}(\emptyset, \beta\rightarrow\Diamond\delta_1)$ and $\mathrm{Pbl}_{\textbf{GL}}(\emptyset, \gamma\rightarrow\Box\neg\delta_1)$. Assume that $\neg\mathrm{Pbl}_{\textbf{GL}}(\emptyset, \gamma\rightarrow\Box\neg\delta_0)$. Then, by Lemma~\ref{select}, $\mathrm{Pbl}_{\textbf{GL}}(\emptyset, \gamma\rightarrow\Diamond\delta_0)$ holds. By $\beta\xrightarrow{m}\gamma$, $\mathrm{Pbl}_{\textbf{GL}}(\emptyset, \beta\rightarrow\Diamond\delta_0)$ holds. It is a contradiction that $\mathrm{Pbl}_{\textbf{GL}}(\emptyset, \beta\rightarrow\Box\neg\delta_0)$ holds and $\beta$ is \textbf{GL}-consistent. Thus $\mathrm{Pbl}_{\textbf{GL}}(\emptyset, \gamma\rightarrow\Box\neg\delta_0)$ holds. 

(2): By Proposition~\ref{fin} and the definition of $\xrightarrow{d}$.
\end{proof}

\begin{prop}[$\mathrm{RCA}_0$]\label{gl4}
Let $\alpha, \beta\in\tilde{W}$ and $\mathrm{\Sigma}=\{L\}$. Then $\alpha\xrightarrow{c}\beta$ implies $\alpha\xrightarrow{m}\beta$.
\end{prop}

\begin{proof}
By Proposition~\ref{crm} and the fact that $\mathrm{Pbl}_{\mathbf{GL}}(\emptyset, 4)$ holds. 
\end{proof}

\begin{lem}[Truth lemma for $\mathbf{GL}$, $\mathrm{RCA}_0$]
Let $\varphi\in\mathcal{L}_{h, n}$ be $\mathbf{GL}$-consistent. Then $(\tilde{W}, \xrightarrow{d}, \tilde{V})$ is a weak Kripke model.
\end{lem}

\begin{proof}

We show that for all $\alpha\in\tilde{W}$, $\tilde{V}(\alpha, \psi)$ satisfies the conditions of a valuation by induction on the length of $\psi$, so we can prove it in $\mathrm{I}\mathrm{\Sigma}^0_1$ (Note that $(\tilde{W}, \xrightarrow{d}, \tilde{V})$ is finite). We show it in the case of $\psi\equiv\Box\theta$, that is, for all $\alpha\in\tilde{W}$,
\begin{equation*}
\tilde{V}(\alpha, \Box\theta)=1\iff\forall \beta\in \tilde{W}(\alpha\xrightarrow{d}\beta\Rightarrow\tilde{V}(\beta, \theta)=1).
\end{equation*}

Fix $\alpha\in\tilde{W}$. We have two cases. First, we assume that $\tilde{V}(\alpha, \Box\theta)=0$, i.e., $\neg\mathrm{Pbl}_{\textbf{GL}}(\emptyset, \alpha\rightarrow\Box\theta)$ holds. Then by Lemma~\ref{select}, $\mathrm{Pbl}_{\textbf{GL}}(\emptyset, \alpha\rightarrow\Diamond\neg\theta)$ holds. Moreover, by the axiom $L$, $\neg\mathrm{Pbl}_{\textbf{GL}}(\emptyset, \alpha\rightarrow\Box(\Box\theta\rightarrow\theta))$ holds. Thus $\alpha\land\Diamond(\Box\theta\land\neg\theta)$ is \textbf{GL}-consistent. So by Lemma~\ref{exlem}, there exists $\beta$ such that $\mathrm{Pbl}_{\textbf{GL}}(\emptyset, \beta\rightarrow\Box\theta\land\neg\theta)$ holds. Then, by Proposition~\ref{gl4}, $\alpha\xrightarrow{m}\beta$, and $\mathrm{Pbl}_{\textbf{GL}}(\emptyset, \beta\rightarrow\Box\neg\neg\theta)$ holds. Thus $\alpha\xrightarrow{d}\beta$. Moreover, $\beta$ is \textbf{GL}-consistent since $\alpha\land\Diamond\beta$ is \textbf{GL}-consistent. Thus $\neg\mathrm{Pbl}_{\textbf{GL}}(\emptyset, \beta\rightarrow\theta)$ holds, so $\tilde{V}(\beta, \theta)=0$. 

The other case is when $\tilde{V}(\alpha, \Box\theta)=1$, that is, $\mathrm{Pbl}_{\textbf{GL}}(\emptyset, \alpha\rightarrow\Box\theta)$ holds.  Fix $\beta\in\tilde{W}$ such that $\alpha\xrightarrow{d}\beta$. Then $\mathrm{Pbl}_{\textbf{GL}}(\emptyset, \alpha\rightarrow\Diamond\beta^{\prime})$ holds. Assume that $\mathrm{Pbl}_{\textbf{GL}}(\emptyset, \beta^{\prime}\rightarrow\neg\theta)$ holds. Then $\mathrm{Pbl}_{\textbf{GL}}(\emptyset, \Diamond\beta^{\prime}\rightarrow\neg\Box\theta)$ holds, thus $\mathrm{Pbl}_{\textbf{GL}}(\emptyset, \alpha\rightarrow\neg\Box\theta)$ holds. It is the contradiction that $\alpha$ is \textbf{GL}-consistent. Thus $\mathrm{Pbl}_{\textbf{GL}}(\emptyset, \beta^{\prime}\rightarrow\theta)$ holds. Then $\mathrm{Pbl}_{\textbf{GL}}(\emptyset, \beta\rightarrow\beta^{\prime})$ holds, so $\mathrm{Pbl}_{\textbf{GL}}(\emptyset, \beta\rightarrow\theta)$ holds. 
\end{proof}

From the above, we obtain the following result. 

\begin{thm}[Weak Completeness Theorem for \textbf{GL}, $\mathrm{RCA}_0$]
If $\varphi\in\mathrm{Fml}$ is $\mathbf{GL}$-consistent, then there exist a frame $F$ which is appropriate to $L$ and a valuation $V$ such that $(F, V)$ is a weak model of $\varphi$.
\end{thm}

\section{Strong Completeness Theorem using canonical models in second-order arithmetic}
In this section, we will prove that the strong completeness theorem for modal logic using canonical models is equivalent to $\mathrm{ACA}_0$ over $\mathrm{RCA}_0$. In this section, we will assume $\mathrm{\Sigma}\subseteq\{T, B, 4, 5, D, .2\}$ unless otherwise specified.

Basically, we still formalize the strong completeness proofs using canonical models.
However, we cannot collect all completions of a given set of formulas within second-order arithmetic.
Therefore, we use the well-known idea of the ``low completion'' from the computability theory, and make it possible to perform the construction in $\mathrm{ACA}_0$.

In $\mathrm{ACA}_{0}$, we can define some notions for computability theory with Turing jumps (see Simpson \cite{MR2517689} for the details).
Let $X, Y\subseteq\mathbb{N}$. Then, $Y$ is said to be \emph{low relative to $X$} if $\mathrm{TJ}(Y)\leq_{T}\mathrm{TJ}(X)$ (here $\mathrm{TJ}(X)$ denotes the Turing jump of $X$).
If $Y\leq_{T}\mathrm{TJ}(X)$, we may consider a $\mathrm{\mathrm{\Delta}}^{0, X}_2$-index of $Y$ (obtained as a $\mathrm{TJ}(X)$-recursive index of $Y$). Moreover, the statement ``$e$ is a $\mathrm{TJ}(X)$-recursive index of a set'' is arithmetical, and in such a case, we write $\mathrm{\mathrm{\Delta}}^{0, X}_2[e]$ for the set indexed by $e$.
\begin{dfn}[$\mathrm{ACA}_0$]
Let $X\subseteq\mathbb{N}$.
We define the set $L(X)$ by the following:
\begin{equation*}
\begin{split}
L(X)=&\{e\in\mathbb{N}\mid\text{$e$ is a $\mathrm{\mathrm{\Delta}}^{0, X}_2$-index and $\mathrm{\Delta}^{0, X}_2[e]$ is a low set relative to $X$.}\}
\end{split}
\end{equation*}
\end{dfn}
Note that within $\mathrm{ACA}_{0}$ the set $\mathrm{\Delta}^{0, X}_2[e]$ always exists, and thus we usually identify the set $\mathrm{\Delta}^{0, X}_2[e]$ and its index $e$.

\begin{lem}[Low Lindenbaum's Lemma, $\mathrm{ACA}_0$]\label{lindenbaum}
If $\mathrm{\Gamma}\subseteq\mathrm{Fml}$ is $\mathbf{K}\mathrm{\mathrm{\Sigma}}$-consistent, then there exists a low relative to $\mathrm{\Gamma}$ set $\mathrm{\Gamma}^{+}\supseteq\mathrm{\Gamma}$ which is maximally $\mathbf{K}\mathrm{\mathrm{\Sigma}}$-consistent.
\end{lem}
\begin{proof}
By the low basis theorem and the usual proof of Lindenbaum's lemma. See Simpson \cite{MR2517689}.
\end{proof}

\begin{lem}[$\mathrm{ACA}_0$]\label{vvv}
Let $\psi$ be a formula, $\mathrm{\Gamma}$ be $\mathbf{K}\mathrm{\mathrm{\Sigma}}$-consistent, and $E$ be the maximally $\mathbf{K}\mathrm{\mathrm{\Sigma}}$-consistent $low$ relative to $\mathrm{\Gamma}$ set. If $\Box\psi\not\in E$, then there exists a maximally $\mathbf{K}\mathrm{\mathrm{\Sigma}}$-consistent $low$ relative to $\mathrm{\Gamma}$ set $\tilde{E}$ such that $\{\neg\psi\}\cup\{\xi|\Box\xi\in E\}\cup\mathrm{\Gamma}\subseteq\tilde{E}$.
\end{lem}

\begin{proof}
Suppose $\Box\psi\not\in E$. We show that $\{\neg\psi\}\cup\{\xi\mid\Box\xi\in E\}\cup\mathrm{\Gamma}$ is $\mathbf{K}\mathrm{\mathrm{\Sigma}}$-consistent. Assume that $\{\neg\psi\}\cup\{\xi\mid\Box\xi\in E\}\cup\mathrm{\Gamma}$ is $\mathbf{K}\mathrm{\mathrm{\Sigma}}$-inconsistent, then there exist $\Box\xi_1, \dots\Box\xi_n\in E$ such that $\mathrm{Pbl}_{\mathbf{K}\mathrm{\mathrm{\Sigma}}}(\mathrm{\Gamma}, (\xi_1\land\dots\land\xi_n\land\neg\psi)\rightarrow\bot)$, that is, $\mathrm{Pbl}_{\mathbf{K}\mathrm{\mathrm{\Sigma}}}(\mathrm{\Gamma}, (\xi_1\land\dots\land\xi_n)\rightarrow\psi)$ holds.  Thus $\mathrm{Pbl}_{\mathbf{K}\mathrm{\mathrm{\Sigma}}}(\mathrm{\Gamma}, (\Box\xi_1\land\dots\land\Box\xi_n)\rightarrow\Box\psi)$ holds. Thus $\Box\psi\in E$, it is a contradiction. Therefore by Low Lindenbaum's Lemma, there exists a low relative to $\mathrm{\Gamma}$ set $\tilde{E}$ such that $\{\neg\psi\}\cup\{\xi|\Box\xi\in E\}\cup\mathrm{\Gamma}\subseteq\tilde{E}$.
\end{proof}

\begin{dfn}[$\mathrm{ACA}_0$]\label{cano}
Let $\mathrm{\Gamma}$ be $\mathbf{K}\mathrm{\mathrm{\Sigma}}$-consistent. A $\mathbf{K}\mathrm{\mathrm{\Sigma}}$-model $M^{\mathbf{K}\mathrm{\mathrm{\Sigma}}}$ is a tuple $(W^{\mathbf{K}\mathrm{\mathrm{\Sigma}}}, R^{\mathbf{K}\mathrm{\mathrm{\Sigma}}}, V^{\mathbf{K}\mathrm{\mathrm{\Sigma}}})$ satisfying the following conditions:
\begin{itemize}
    \item $W^{\mathbf{K}\mathrm{\mathrm{\Sigma}}}=\{e\in L(\mathrm{\Gamma})\mid\mathrm{\Gamma}\subseteq\mathrm{\Delta}^{0, X}_2[e] \text{ and $\mathrm{\Delta}^{0, X}_2[e]$ is maximally $\mathbf{K}\mathrm{\mathrm{\Sigma}}$-consistent.}\}$,
    \item $R^{\mathbf{K}\mathrm{\mathrm{\Sigma}}}=\{(e, \tilde{e})\in W^{\mathbf{K}\mathrm{\mathrm{\Sigma}}}\times W^{\mathbf{K}\mathrm{\mathrm{\Sigma}}}\mid\forall\varphi\in \mathrm{Fml}(\Box\varphi\in \mathrm{\Delta}^{0, X}_2[e]\rightarrow\varphi\in\mathrm{\Delta}^{0, X}_2[\tilde{e}]\}$,
    \item $V^{\mathbf{K}\mathrm{\mathrm{\Sigma}}}(e, \varphi)=1\iff\varphi\in \mathrm{\Delta}^{0, X}_2[e]$ for any $e\in W^{\mathbf{K}\mathrm{\mathrm{\Sigma}}}$ and any $\varphi\in\mathrm{Fml}$.
\end{itemize}
\end{dfn}

To simplify the notations, we may write $E, E', \tilde{E}$,\dots for the sets $\mathrm{\Delta}^{0, X}_2[e]$, $\mathrm{\Delta}^{0, X}_2[e']$, $\mathrm{\Delta}^{0, X}_2[\tilde{e}]$,\dots.

\begin{prop}[$\mathrm{RCA}_0$]\label{esp}
Let $\varphi, \psi\in\mathrm{Fml}$. Let $E$ be maximally $\mathbf{K}\mathrm{\mathrm{\Sigma}}$-consistent. Then the following holds:
\begin{enumerate}
    \item $\neg\varphi\in E\iff\varphi\not\in E$,
    \item $\varphi\rightarrow\psi\in E\iff(\varphi\not\in E\lor \psi\in E)$,
    \item $\varphi\land\psi\in E\iff\varphi\in E\land \psi\in E$.
\end{enumerate}
\end{prop}

\begin{proof}
It is clear since $E$ is $\mathbf{K}\mathrm{\mathrm{\Sigma}}$-consistent.
\end{proof}

\begin{prop}[$\mathrm{RCA}_0$]\label{zimei}
For $\mathbf{K}\mathrm{\mathrm{\Sigma}}$-model $M^{\mathbf{K}\mathrm{\mathrm{\Sigma}}}$, the following holds:
\begin{equation*}
R^{\mathbf{K}\mathrm{\mathrm{\Sigma}}}=\{(e, e^{\prime})\in W\times W\mid\forall\varphi\in\mathrm{Fml}(\varphi\in \tilde{E}\rightarrow\Diamond\varphi\in E)\}
\end{equation*}
\end{prop}

\begin{proof}
It is clear by the definition of $R^{\mathbf{K}\mathrm{\mathrm{\Sigma}}}$.
\end{proof}

\begin{lem}[$\mathrm{ACA}_0$]\label{dependson}
Let $W=W^{\mathbf{K}\mathrm{\mathrm{\Sigma}}}$ and $R=R^{\mathbf{K}\mathrm{\mathrm{\Sigma}}}$. Then $(W, R)$ is appropriate to $\mathrm{\Sigma}$.
\end{lem}
\begin{proof}
We will check frame condtions.
Let $\varphi\in\mathrm{Fml}$, and let $e\in W$.

($T$): Let $e\in W$. Assume that $\Box\varphi\in E$. Then by $T\in\mathrm{\Sigma}$, that is, $T\in E$, $\Box\varphi\rightarrow\varphi\in E$, so $\varphi\in E$. Thus $eRe$ holds.

($B$): Let $eR\tilde{e}$. Assume that $\varphi\in E$, then by $B\in\mathrm{\Sigma}$, that is, $B\in E$, $\varphi\rightarrow\Box\Diamond\varphi\in E$, so $\Box\Diamond\varphi\in E$. By $eR\tilde{e}$, $\Diamond\varphi\tilde{E}$. Thus $\tilde{e}Re$ holds.

($4$): Let $e_0, e_1, e_2\in W$ such that $e_0Re_1$ and $e_1Re_2$. Assume that $\Box\varphi\in E_0$, then by $4\in\mathrm{\Sigma}$, that is, $4\in E_0$, so $\Box\Box\varphi\in E_0$. Then $\varphi\in E_2$. Thus $e_0Re_2$ holds.

($5$): Let $e_0, e_1, e_2\in W$ such that $e_0Re_1$ and $e_0Re_2$. Assume that $\varphi\in E_2$, then $\Diamond\varphi\in E_0$. By $5\in\mathrm{\Sigma}$, that is, $5\in E_0$, $\Box\Diamond\varphi\in E_0$. Thus $\Diamond\varphi\in E_1$. Thus $e_1Re_2$ holds.

($D$): Let $e\in W$. Note that $\top\in E$. Then $\Box\neg\bot\in E$. By $D\in\mathrm{\Sigma}$, that is, $D\in E$, $\neg\Box\bot\in E$. Then $\Box\bot\not\in E$. By Lemma~\ref{vvv}, there exists $\tilde{e}\in W$ such that $\{\top\}\cup\{\xi\mid\Box\xi\in E\}\cup\mathrm{\Gamma}\subseteq\tilde{E}$. Thus $eR\tilde{e}$ holds.

($.2$): Let $e_0, e_1, e_2\in W$ such that $e_0Re_1$ and $e_1Re_2$. We show that $\{\xi\mid\Box\xi\in E_1\}\cup\{\delta\mid\Box\delta\in E_2\}$ is $\mathbf{K}\mathrm{\mathrm{\Sigma}}$-consistent. Assume that $\{\xi\mid\Box\xi\in E_1\}\cup\{\delta\mid\Box\delta\in E_2\}$ is $\mathbf{K}\mathrm{\mathrm{\Sigma}}$-inconsistent, there exist some $\Box\xi_1, \dots, \Box\xi_n\in E_1$ and $\Box\delta_1, \dots, \Box\delta_m\in E_2$ such that $\mathrm{Pbl}_{\mathbf{K}\mathrm{\mathrm{\Sigma}}}(\emptyset, (\xi_1\land\dots\land\xi_n\land\delta_1\land\dots\land\delta_m)\rightarrow\bot)$ holds. Since $\Box(\xi_1\land\dots\land\xi_n)\in E_1$ and $e_0Re_1$, $\Diamond\Box(\xi_1\land\dots\land\xi_n)\in E_0$. On the other hands, $\mathrm{Pbl}_{\mathbf{K}\mathrm{\mathrm{\Sigma}}}(\emptyset, (\xi_1\land\dots\land\xi_n)\rightarrow\neg(\delta_1\land\dots\land\delta_m))$ holds, so $\mathrm{Pbl}_{\mathbf{K}\mathrm{\mathrm{\Sigma}}}(\emptyset, \Diamond\Box(\xi_1\land\dots\land\xi_n)\rightarrow\Diamond\Box\neg(\delta_1\land\dots\land\delta_m))$ holds. Thus $\mathrm{Pbl}_{\mathbf{K}\mathrm{\mathrm{\Sigma}}}(\emptyset, \Diamond\Box(\xi_1\land\dots\land\xi_n)\rightarrow\neg\Box\Diamond(\delta_1\land\dots\land\delta_m))$ holds, so we have $\neg\Box\Diamond(\delta_1\land\dots\land\delta_m)\in E_0$. However, by ${.2}\in\mathrm{\Sigma}$, that is, ${.2}\in E_0$ and $\Diamond\Box(\delta_1\land\dots\land\delta_m)\in E_0$, $\Box\Diamond(\delta_1\land\dots\land\delta_m)\in E_0$, it is a contradiction that $E_0$ is $\mathbf{K}\mathrm{\mathrm{\Sigma}}$-consistent. Thus $\{\xi\mid\Box\xi\in E_1\}\cup\{\delta\mid\Box\delta\in E_2\}$ is $\mathbf{K}\mathrm{\mathrm{\Sigma}}$-consistent. By Lemma~\ref{vvv}, there exists $e_3\in W$ such that $\{\xi\mid\Box\xi\in E_1\}\cup\{\delta\mid\Box\delta\in E_2\}\cup\mathrm{\Gamma}\subseteq E_3$. So $e_1Re_3$ and $e_2Re_3$ hold.
\end{proof}

\begin{lem}[Truth Lemma, $\mathrm{ACA}_0$]\label{truthlem}
The truth valuation $V^{\mathbf{K}\mathrm{\mathrm{\Sigma}}}$ satisfies the following conditions: 
\begin{itemize}
    \item $V^{\mathbf{K}\mathrm{\mathrm{\Sigma}}}(w, \bot)=0$ for any $w\in W^{\mathbf{K}\mathrm{\mathrm{\Sigma}}}$, 
    \item $V^{\mathbf{K}\mathrm{\mathrm{\Sigma}}}(w, \varphi\rightarrow\psi)=1-V^{\mathbf{K}\mathrm{\mathrm{\Sigma}}}(w, \varphi)(1-V^{\mathbf{K}\mathrm{\mathrm{\Sigma}}}(w, \psi))$ for any $w\in W^{\mathbf{K}\mathrm{\mathrm{\Sigma}}}$ and any $\varphi, \psi\in\mathrm{Fml}$, 
    \item $V^{\mathbf{K}\mathrm{\mathrm{\Sigma}}}(w, \Box\varphi)=1\iff\forall v\in W^{\mathbf{K}\mathrm{\mathrm{\Sigma}}}(wR^{\mathbf{K}\mathrm{\mathrm{\Sigma}}}v\rightarrow V^{\mathbf{K}\mathrm{\mathrm{\Sigma}}}(v, \varphi)=1)$ for any $w\in W^{\mathbf{K}\mathrm{\mathrm{\Sigma}}}$ and any $\varphi\in\mathrm{Fml}$.
\end{itemize}
\end{lem}
\begin{proof}
We show this lemma by induction on the length of $\varphi$ using $\mathrm{ACA}_0$. Fix $e\in W$. If $\varphi\equiv\bot$, it is clear.

Assume that $\varphi\equiv\psi\rightarrow\theta$. Then one of (1) $V^{\mathbf{K}\mathrm{\mathrm{\Sigma}}}(e, \psi)=0\iff\psi\not\in E$ or (2) $V^{\mathbf{K}\mathrm{\mathrm{\Sigma}}}(e, \theta)=1\iff\theta\in E$ or (3) $(V^{\mathbf{K}\mathrm{\mathrm{\Sigma}}}(e, \psi)=1\land V^{\mathbf{K}\mathrm{\mathrm{\Sigma}}}(e, \theta)=0)\iff(\psi\in E\land\theta\not\in E)$ is satisfiable. If either (1) or (2) holds, then by Proposition~\ref{esp}, $\psi\rightarrow\theta\in E\iff V^{\mathbf{K}\mathrm{\mathrm{\Sigma}}}(e, \psi\rightarrow\theta)=1$. If (3) holds, then by Proposition~\ref{esp}, $\psi\rightarrow\theta\not\in E\Leftrightarrow V^{\mathbf{K}\mathrm{\mathrm{\Sigma}}}(e, \psi\rightarrow\theta)=0$. Thus $V^{\mathbf{K}\mathrm{\mathrm{\Sigma}}}(w, \psi\rightarrow\theta)=1-V^{\mathbf{K}\mathrm{\mathrm{\Sigma}}}(w, \psi)(1-V^{\mathbf{K}\mathrm{\mathrm{\Sigma}}}(w, \theta))$ holds.

Assume that $\varphi\equiv\Box\psi$. We have two cases. First, we assume that $V^{\mathbf{K}\mathrm{\mathrm{\Sigma}}}(e, \Box\psi)=0$. Then $\Box\psi\not\in E$. Then by Lemma~\ref{vvv}, there exists $\tilde{e}\in W$ such that $\{\neg\psi\}\cup\{\xi|\Box\xi\in E\}\cup\mathrm{\Gamma}\subseteq\tilde{E}$. Thus by Proposition~\ref{zimei}, $eR^{\mathbf{K}\mathrm{\mathrm{\Sigma}}}\tilde{e}$ and $\neg\psi\in\tilde{E}$, so $\psi\not\in\tilde{E}$ holds. Thus $V^{\mathbf{K}\mathrm{\mathrm{\Sigma}}}(\tilde{e}, \psi)=0$. The other case is when $V^{\mathbf{K}\mathrm{\mathrm{\Sigma}}}(e, \Box\psi)=1$, that is, $\Box\psi\in E$ holds. Fix $\tilde{e}\in W$ with $eR^{\mathbf{K}\mathrm{\mathrm{\Sigma}}}\tilde{e}$. Then $\psi\in\tilde{E}$, that is, $V^{\mathbf{K}\mathrm{\mathrm{\Sigma}}}(\tilde{e}, \psi)=1$. 
\end{proof}

From the above, we obtain the strong completeness theorem using a canonical model.

\begin{thm}[Strong Completeness Theorem using canonical models, $\mathrm{ACA}_0$]\label{SCT}
We assume that $\mathrm{\Sigma}\subseteq\{T, B, 4, 5, D, .2\}$. Let $\mathrm{\Gamma}$ be a set of fomulas. If $\mathrm{\Gamma}$ is $\mathbf{K}\mathrm{\mathrm{\Sigma}}$-consistent, then there exist a frame $F=(W, R)$ which is appropriate to $\mathrm{\Sigma}$ and a valuation $V$ such that $M=(F, V)$ is a model of $\mathrm{\Gamma}$ and the following holds:
\begin{equation*}
\forall\varphi\in\mathrm{Fml}(\mathrm{Pbl}_{\mathbf{K}\mathrm{\mathrm{\Sigma}}}(\mathrm{\Gamma}, \varphi)\leftrightarrow\forall e\in W(V(e, \varphi)=1)).
\end{equation*}
\end{thm}
\begin{proof}
Let $W=W^{\mathbf{K}\mathrm{\mathrm{\Sigma}}}$, $R=R^{\mathbf{K}\mathrm{\mathrm{\Sigma}}}$, and $V=V^{\mathbf{K}\mathrm{\mathrm{\Sigma}}}$ by Definition~\ref{cano}. This is the desired Kripke model by Truth Lemma. Moreover, it is clear that $\forall\varphi\in\mathrm{\Gamma}\forall e\in W(V(e, \varphi)=1)$. Let $\varphi\in\mathrm{Fml}$. We have the two cases. If $\mathrm{Pbl}_{\mathbf{K}\mathrm{\mathrm{\Sigma}}}(\mathrm{\Gamma}, \varphi)$, then $\forall e\in W(V(e, \varphi)=1)$ holds by the soundness theorem. If $\neg\mathrm{Pbl}_{\mathbf{K}\mathrm{\mathrm{\Sigma}}}(\mathrm{\Gamma}, \varphi)$, then by Low Lindenbaum's Lemma, there exists $e\in W$ such that $\neg\varphi\in E$, that is, $\varphi\not\in E$. Then $V(e, \varphi)=0$.
\end{proof}

\begin{thm}
The following statements are equivalent over $\mathrm{RCA}_0$:
\begin{enumerate}
    \item $\mathrm{ACA}_0$,
    \item Strong Completeness Theorem using canonical models.
\end{enumerate}
\end{thm}

\begin{proof}
$(1)$ implies $(2)$ by Theorem~\ref{SCT}. We show that $(2)$ implies $(1)$. It is enough to show $\mathrm{\Sigma}^0_1$-$\mathrm{CA}$. Let $\varphi(n)\equiv\exists m\theta(n, m)$ be $\mathrm{\Sigma}^0_1$-fomula ($\theta(n, m)$ is a $\mathrm{\Sigma}^0_0$-formula). Let $\alpha_0, \alpha_1, \dots$ be an enumeration of $\mathrm{C}$. Then we define $\mathrm{\Gamma}=\{\alpha_n\lor\dots\lor\alpha_n\mid\exists m\leq\langle{\alpha_n\lor\dots\lor\alpha_n\rangle}\theta(n, m)\}$. Any finite subset of $\mathrm{\Gamma}$ has a model, so by the Soundness Theorem, any finite subset of $\mathrm{\Gamma}$ is $\mathbf{K}\mathrm{\mathrm{\Sigma}}$-consistent. Thus $\mathrm{\Gamma}$ is $\mathbf{K}\mathrm{\mathrm{\Sigma}}$-consistent. Then by the Strong Completeness Theorem using canonical models and the Soundness Theorem, the following holds.
\begin{equation*}
\varphi(n)\leftrightarrow\alpha_n\lor\dots\lor\alpha_n\in\mathrm{\Gamma}\leftrightarrow\mathrm{Pbl}_{\mathbf{K}\mathrm{\mathrm{\Sigma}}}(\mathrm{\Gamma}, \alpha_n)\leftrightarrow\forall w\in W(V(w, \alpha_n)=1)
\end{equation*}

Thus by $\mathrm{\Delta}^0_1$-$\mathrm{CA}$, there exists a set $X$ such that $\forall n(n\in X\leftrightarrow\varphi(n))$ holds. Thus $\mathrm{\Sigma}^0_1$-$\mathrm{CA}$ holds.
\end{proof}

\section{Strong Completeness Theorem in second-order arithmetic}
We showed Theorem~\ref{SCT} by constructing a canonical model for a given $\mathbf{K}\mathrm{\mathrm{\Sigma}}$-consistent set $\mathrm{\Gamma}$. On the other hand, we consider the following statement: for a given $\mathbf{K}\mathrm{\mathrm{\Sigma}}$-consistent set $\mathrm{\Gamma}\cup\{\psi\}$, there exists a model for $\mathrm{\Gamma}\cup\{\psi\}$. This is a weaker formulation of the strong completeness theorem. First of all, we give a natural first-order logic language $\mathcal{L}_{\mathcal{FO}}$ for the Kripke model of modal logic. Using the standard translation, we aim to obtain the Kripke model from the properties of first-order logic.

\begin{dfn}[$\mathrm{RCA}_0$]
Let $\mathcal{L}_{\mathcal{FO}}$ be the first-order language (with equality) which  has unary predicates $P_0, P_1, \dots$ corresponding to the proposition letters $p_0, p_1, \dots$, and $2$-ary reation symbol $r$.
\end{dfn}

Then, as in the case of modal logic, we can show the existence of the set of $\mathcal{L}_{\mathcal{FO}}$-formulas and axioms of first-order logic in $\mathrm{RCA}_0$, i.e., we can show in $\mathrm{RCA}_0$ that there exist sets $\mathrm{Trm}_{\mathcal{FO}}$, $\mathrm{Fml}_{\mathcal{FO}}$, $\mathrm{Snt}_{\mathcal{FO}}$, and $\mathrm{Axm}_\mathcal{FO}$ consisting of all $\mathcal{L}_{\mathcal{FO}}$-terms, $\mathcal{L}_{\mathcal{FO}}$-fomulas, $\mathcal{L}_{\mathcal{FO}}$-sentences, and axioms of first-order logic in $\mathcal{L}_{\mathcal{FO}}$ respectively. Moreover, as in the case of modal logic, we can define the provability predicate of first-order logic in $\mathrm{RCA}_0$. We assume that the only logical rule is modus ponens (See Simpson \cite{MR2517689}).

\begin{dfn}[$\mathrm{RCA}_0$]
Let $\mathrm{\Sigma}\subseteq\{T, B, 4, D, 5, .2\}$. We define the function $A: \mathrm{\Sigma}\rightarrow\mathrm{Snt}_{\mathcal{FO}}$ by the following:
\begin{equation*}
A(\psi)=
  \begin{dcases*}
  \forall xr(x, x) & if $\psi\equiv T$,\\
  \forall x\forall y(r(x, y)\rightarrow r(y,x)) & if $\psi\equiv B$,\\
  \forall x\forall y\forall z(r(x, y)\land r(y, z)\rightarrow r(x, z)) & if $\psi\equiv 4$,\\
  \forall x\exists yr(x, y) & if $\psi\equiv D$,\\
  \forall x\forall y\forall z(r(x, y)\land r(x, z)\rightarrow r(y, z)) & if $\psi\equiv 5$,\\
  \forall x\forall y\forall z(r(x, y)\land r(x, z)\rightarrow\exists w(r(y, w)\land r(z, w))) & if $\psi\equiv .2$.\\
  \end{dcases*}
\end{equation*}

Moreover, we define $\mathrm{\Sigma}_{\mathcal{FO}}=\{A(\psi)\mid\psi\in\mathrm{\Sigma}\}$.
\end{dfn}

\begin{dfn}[$\mathrm{RCA}_0$]
Let $X\subset\mathrm{Fml}_\mathcal{FO}$ and $\mathrm{\Sigma}\subseteq\{T, B, 4, D, 5, .2\}$. We will define the following fomulas:
\begin{equation*}
\begin{split}
\mathrm{Prf}_{\mathcal{FO}(\mathrm{\Sigma})}(X, p)\equiv p&\in\mathrm{Seq}\land\forall k(k<lh(p)\rightarrow p(k)\in\mathrm{Fml}_\mathcal{FO})\\&\land\forall k\Bigl(k<lh(p)\rightarrow\Bigl(p(k)\in X\lor p(k)\in\mathrm{Axm}_\mathcal{FO}\lor p(k)\in\mathrm{\Sigma}_{\mathcal{FO}}\\&\lor(\exists i<k\exists j<k)(p(i)=p(j)\rightarrow p(k))\Bigr)\Bigr),
\end{split}
\end{equation*}
\begin{equation*}
\begin{split}
\mathrm{Pbl}_{\mathcal{FO}(\mathrm{\Sigma})}(X, \varphi)\equiv&\exists\psi_1,\dots,\psi_n\in X\exists p\\&\bigl(\mathrm{Prf}_{\mathcal{FO}(\mathrm{\Sigma})}(X, p)\land(\exists i<lh(p))(p(i)=(\psi_1\land\dots\land\psi_n\rightarrow\varphi))\bigr).
\end{split}
\end{equation*}
\end{dfn}

Furthermore, as in the case of modal logic, we can use $\mathrm{RCA}_0$ to describe the consistency of the system and the satisfiability of the model in first-order logic.

\begin{dfn}[$\mathrm{RCA}_0$]
Let $M$ be a set. Let $T_M$ and $S_M$ be respectively the sets of closed terms and sentences of the expanded language $\mathcal{L}^{M}_{\mathcal{FO}}=\mathcal{L}_{\mathcal{FO}}\cup\{\overline{m}\mid m\in M\}$ with new constant symbols $\overline{m}$ for each element $m\in M$. 

$\mathcal{M}=(M, v)$ is a standard model on $\mathcal{L}_{\mathcal{FO}}$ if the function $v: T_{M}\cup S_{M}\rightarrow M\cup\{0, 1\}$ satisfies the following conditions:
\begin{itemize}
    \item $\forall t\in T_M(v(t)\in M)$,
    \item $\forall\sigma\in S_M(v(\sigma)\in\{0, 1\})$,
    \item $\forall t_1, \forall t^{\prime}_1, \dots \forall t_n, \forall t^{\prime}_n\in T_M\\\bigl(\forall i\leq n(v(t_i)=v(t^{\prime}_i))\rightarrow v(R(t_1,\dots,t_n))=v(R(t^{\prime}_1,\dots,t^{\prime}_n))\bigr)$, where $R$ is a relation symbol,
    \item $v(\neg\sigma)=1-v(\sigma)$,
    \item $v(\sigma_0\land\sigma_1)=v(\sigma_0)\cdot v(\sigma_1)$,
    \item $v(\forall x\varphi(x))=1\iff\forall a\in M(v(\varphi(\overline{a}))=1)$,
    \item $v(\exists x\varphi(x))=1\iff\exists a\in M(v(\varphi(\overline{a}))=1)$.
\end{itemize}
\end{dfn}

Let $X\subseteq\mathrm{Snt}_\mathcal{FO}$. Then a standard model $\mathcal{M}$ is  a model of $X$ on $\mathcal{L}_{\mathcal{FO}}$ if $\forall\varphi\in X(v(\varphi)=1)$ holds.

\begin{dfn}[$\mathrm{RCA}_0$]
Let $\varphi\in\mathrm{Snt}_{\mathcal{FO}}$ and $X\subseteq\mathrm{Snt}_{\mathcal{FO}}$. Let $\mathcal{M}=(M, v)$ be a standard model on $\mathcal{L}_{\mathcal{FO}}$. We write $\mathcal{M}\vDash\varphi$ if $M$ satisfies $\varphi$, i.e., $v(\varphi)=1$, and we write $X\vDash\varphi$ if every model of $X$ on $\mathcal{L}_{\mathcal{FO}}$ satisfies $\varphi$.
\end{dfn}

\begin{thm}[Soundness Theorem for first-order logic, $\mathrm{RCA}_0$ \cite{MR2517689}]
Let \\$\mathrm{\Sigma}\subseteq\{T, B, 4, D, 5, .2\}$ and $\varphi\in\mathrm{Snt}_{\mathcal{FO}}$. If $\mathrm{Pbl}_{\mathcal{FO}(\mathrm{\Sigma})}(\emptyset, \varphi)$ holds, then $\vDash\varphi$ holds.
\end{thm}

\begin{thm}[G\"odel's Completeness Theorem, $\mathrm{WKL}_0$ \cite{MR2517689}]
Let \\$\varphi\in\mathrm{Snt}_{\mathcal{FO}}$ and $\mathrm{\Sigma}\subseteq\{T, B, 4, D, 5, .2\}$. If $X\subseteq\mathrm{Snt}_{\mathcal{FO}}$ is consistent, then $\mathrm{Pbl}_{\mathcal{FO}(\mathrm{\Sigma})}(X, \varphi)$ holds if and only if $X\vDash\varphi$ holds.
\end{thm}

\begin{thm}[\cite{MR2517689}]\label{wgcp}
The following statements are equivalent over $\mathrm{RCA}_0$:
\begin{enumerate}
    \item $\mathrm{WKL}_0$,
    \item G\"odel's Completeness Theorem,
    \item The completeness theorem for propositional logic with countably many atoms.
\end{enumerate}
\end{thm}

The following is essentially a first-order reformulation of the modal satisfacation definition.

\begin{dfn}[$\mathrm{RCA}_0$ \cite{MR1837791}]
Let $x$ be a first-order variable. The \emph{standard translation} $ST_{x}$ taking modal formulas to $\mathcal{L}_{\mathcal{FO}}$-formulas is defined as follows:
\begin{itemize}
    \item $ST_{x}(p_{n})\equiv P_{n}(x)$,
    \item $ST_{x}(\bot)\equiv\bot$,
    \item $ST_{x}(\varphi\rightarrow\psi)\equiv ST_{x}(\varphi)\rightarrow ST_{x}(\psi)$,
    \item $ST_{x}(\Box\varphi)\equiv\forall y(r(x, y)\rightarrow ST_{y}(\varphi))$   ($y$ is fresh).
\end{itemize}
\end{dfn}

\begin{lem}[$\mathrm{RCA}_0$]\label{frightk}
Let $\varphi\in\mathrm{Fml}$. For all $\mathcal{M}=(M, v)$ which is a standard model on $\mathcal{L}_{\mathcal{FO}}$, there exists a Kripke model $M^{\star}$ such that
\begin{enumerate}
    \item for all $w\in M$, $M^{\star}, w\Vdash\varphi\iff\mathcal{M}\vDash ST_{w}(\varphi)$, and
    \item $M^{\star}\Vdash\varphi\iff\mathcal{M}\vDash\forall xST_{x}(\varphi)$.
\end{enumerate}
\end{lem}

\begin{proof}
We define $M^{\star}=(W^{\star}, R^{\star}, V^{\star})$ by the following:
\begin{itemize}
    \item $W^{\star}=M$,
    \item $R^{\star}=r$,
    \item $V^{\star}(w, \varphi)=v(ST_{w}(\varphi))$ for all $w\in W$ and $\varphi\in\mathrm{Fml}$.
\end{itemize}

We can show that $V^{\star}$ is a valuation with $F^{\star}=(W^{\star}, R^{\star})$ by the definition of $v$ and the standard translation. For example, for all $w\in W^{\star}$ and $\Box\varphi\in\mathrm{Fml}$,
\begin{equation*}
\begin{split}
V^{\star}(w, \Box\varphi)=1&\iff v(ST_{w}(\Box\varphi))=1\\&\iff v(\forall u(r(w, u)\rightarrow ST_{u}(\varphi)))=1\\&\iff\forall u\in M(v(r(w, u)\rightarrow ST_{u}(\varphi))=1)\\&\iff\forall u\in M(v(r(w, u))\cdot(1-v(ST_{u}(\varphi)))=0)\\&\iff\forall u\in M(v(r(w, u))=0\lor v(ST_{u}(\varphi))=1)\\&\iff\forall u\in W^{\star}(wRu\rightarrow V^{\star}(u, \varphi)=1).
\end{split}
\end{equation*}

Thus (1) holds by the definition of $V$. (2) also holds from (1).
\end{proof}

\begin{lem}[$\mathrm{RCA}_0$]\label{srightf}
Let $\mathrm{\Sigma}\subseteq\{T, B, 4, D\}$ and $\mathrm{\Gamma}\subseteq\mathrm{Fml}$ be $\mathbf{K}\mathrm{\mathrm{\Sigma}}$-consistent. Then 1 implies 2.
\begin{enumerate}
    \item $\mathrm{Pbl}_{\mathcal{FO}(\mathrm{\Sigma})}(\{\forall xST_{x}(\theta)\mid\theta\in\mathrm{\Gamma}\},\forall xST_{x}(\varphi))$,
    \item $\mathrm{Pbl}_{\mathbf{K}\mathrm{\mathrm{\Sigma}}}(\mathrm{\Gamma}, \varphi)$.
\end{enumerate}
\end{lem}

\begin{proof}
Let $\mathrm{\Gamma}^{*}=\{\forall xST_{x}(\theta)\mid\theta\in\mathrm{\Gamma}\}$ and $\varphi^{*}\equiv\forall xST_{x}(\varphi)$. Assume that $\neg\mathrm{Pbl}_{\mathbf{K}\mathrm{\mathrm{\Sigma}}}(\mathrm{\Gamma}, \varphi)$ and $\mathrm{Pbl}_{\mathcal{FO}(\mathrm{\Sigma})}(\mathrm{\Gamma}^{*},\varphi^{*})$ holds. Then there exists a finite set $\mathrm{\Gamma}_0\subseteq\mathrm{\Gamma}$ such that $\mathrm{Pbl}_{\mathcal{FO}(\mathrm{\Sigma})}(\mathrm{\Gamma}^{*}_0,\varphi^{*})$ holds. By the assumption, $\neg\mathrm{Pbl}_{\mathbf{K}\mathrm{\mathrm{\Sigma}}}(\mathrm{\Gamma}_0, \varphi)$ holds. Since $\bigwedge\mathrm{\Gamma}_0\land\neg\varphi$ is $\mathbf{K}\mathrm{\mathrm{\Sigma}}$-consistent, by the Weak Completeness Theorem, there exist a finite frame $F=(W, R)$ which is appropriate to $\mathrm{\Sigma}$ and a valuation $V$ such that $(F, V)$ is a weak model of $\bigwedge\mathrm{\Gamma}_0\land\neg\varphi$. Then we can define a model of $\bigwedge\mathrm{\Gamma}^{*}_0\land\neg\varphi^{*}$ on $\mathcal{L}_{\mathcal{FO}}$ from $(F, V)$. It is a contradiction by Soundness Theorem for first-order logic.
\end{proof}

\begin{thm}[Strong Completeness Theorem, $\mathrm{WKL}_0$]\label{wfsct}
Suppose \\$\mathrm{\Sigma}\subseteq\{T, B, 4, D\}$. Let $\mathrm{\Gamma}$ be a set of fomulas. Let $\varphi\in\mathrm{Fml}$. If $\mathrm{\Gamma}$ is $\mathbf{K}\mathrm{\mathrm{\Sigma}}$-consistent, then $\mathrm{Pbl}_{\mathbf{K}\mathrm{\mathrm{\Sigma}}}(\mathrm{\Gamma}, \varphi)$ holds if and only if $\mathrm{\Gamma}\Vdash_{\mathbb{F}_{\mathrm{\Sigma}}}\varphi$ holds.
\end{thm}

\begin{proof}
By Soundness Theorem (for modal logic), it is enough to show that \\$\neg\mathrm{Pbl}_{\mathbf{K}\mathrm{\mathrm{\Sigma}}}(\mathrm{\Gamma}, \varphi)$ implies $\mathrm{\Gamma}\not\Vdash_{\mathbb{F}_{\mathrm{\Sigma}}}\varphi$. Assume that $\neg\mathrm{Pbl}_{\mathbf{K}\mathrm{\mathrm{\Sigma}}}(\mathrm{\Gamma}, \varphi)$. Let $\mathrm{\Gamma}^{*}=\{\forall xST_{x}(\theta)\mid\theta\in\mathrm{\Gamma}\}$ and $\varphi^{*}\equiv\forall xST_{x}(\varphi)$. By Lemma~\ref{srightf}, $\neg\mathrm{Pbl}_{\mathcal{FO}(\mathrm{\Sigma})}(\mathrm{\Gamma}^{*},\varphi^{*})$ holds. Then $\mathrm{\Gamma}^{*}$ is consistent since $\mathrm{\Gamma}$ is $\mathbf{K}\mathrm{\mathrm{\Sigma}}$-consistent. Thus by G\"odel's Completeness Theorem, there exists a standard model $\mathcal{M}=(M, v)$ such that $\mathcal{M}$ is a model of $\mathrm{\Gamma}^{*}\cup\{\neg\varphi^{*}\}$ on $\mathcal{L}_{\mathcal{FO}}$. Therefore by Lemma~\ref{frightk}, there exists a Kripke model $M^{\star}=(F^{\star}, V^{\star})$ such  that $M^{\star}\Vdash\mathrm{\Gamma}\cup\{\neg\varphi\}$. Moreover, for all $\mathrm{\Psi}\in\mathrm{\Sigma}_{\mathcal{FO}}$, $\mathcal{M}\vDash\mathrm{\Psi}$, so $F^{\star}\Vdash\mathrm{\Sigma}$. Thus $\mathrm{\Gamma}\not\Vdash_{\mathbb{F}_{\mathrm{\Sigma}}}\varphi$. 
\end{proof}

\begin{thm}\label{wklwfsct}
The following statements are equivalent over $\mathrm{RCA}_0$:
\begin{enumerate}
    \item $\mathrm{WKL}_0$,
    \item Strong Completeness Theorem.
\end{enumerate}
\end{thm}
\begin{proof}
(1) implies (2) from Theorem~\ref{wfsct}. Clearly, Strong Completeness Theorem implies the completeness theorem for propositional logic with countably many atoms. Thus by Theorem~\ref{wgcp}, we have $\mathrm{WKL}_0$.
\end{proof}

\section{Open questions}

Some questions remain open concernig the weak completeness theorem in second-order arithmetic. In particular, the completeness with respect to the modal logic cube is not yet fully known. We showed Weak Completeness Theorem for the case of $\mathrm{\Sigma}\subseteq\{T, B, 4, D\}$ in $\mathrm{RCA}_0$. However, we have not shown Weak Completeness Theorem for the case of $5\in\mathrm{\Sigma}$:

\begin{que}
Is Weak Completeness Theorem for the case of $5\in\mathrm{\Sigma}$ provable in $\mathrm{RCA}_0$?
\end{que}

In fact, the following weak model $M=(W, \xrightarrow{s}, V)$ is the model for cases $\mathrm{\Sigma}=\{4, 5\}$, $\mathrm{\Sigma}=\{4, 5, D\}$.

\begin{itemize}
    \item $W=\{\alpha\in C_{h+1, n}\mid\neg\mathrm{Pbl}_{\mathbf{K}\mathrm{\mathrm{\Sigma}}}(\alpha,\bot)\land\mathrm{Pbl}_{\mathbf{K}\mathrm{\mathrm{\Sigma}}}(\emptyset, \alpha\rightarrow\varphi)\}$,
    \item $\xrightarrow{s}=\{(\alpha, \beta)\in W\times W\mid\alpha\xrightarrow{m}\beta\land\forall\xi\in C_{h, n}(\mathrm{Pbl}_{\mathbf{K}\mathrm{\mathrm{\Sigma}}}(\emptyset, \alpha\rightarrow\Diamond\xi)\rightarrow\mathrm{Pbl}_{\mathbf{K}\mathrm{\mathrm{\Sigma}}}(\emptyset, \beta\rightarrow\Diamond\xi))\}$,
    \item For all $\psi\in sub(\varphi)$ and $\alpha\in W$, $V(\alpha, \psi)=1\iff\mathrm{Pbl}_{\mathbf{K}\mathrm{\mathrm{\Sigma}}}(\emptyset, \alpha\rightarrow\psi)$.
\end{itemize}

However, it is known that in the case of $5\in\mathrm{\Sigma}$ and $4\not\in\mathrm{\Sigma}$, this model does not work for $h=0$ and $n=1$ by Moss \cite{MR2349876}. We expect that the $\mathrm{\Sigma}=\{5\}$ and $\mathrm{\Sigma}=\{5, D\}$ cases will force us to construct different models.

\begin{que}
Is Weak Completeness Theorem for the case of $\mathrm{\Sigma}=\{5\}$ and $\mathrm{\Sigma}=\{5, D\}$ provable in $\mathrm{RCA}_0$?
\end{que}

The modal logic \textbf{S4.2} is the smallest normal modal logic which contains $T, 4$, and $.2$. \textbf{S4.2} is the important object that appears in the analysis of modal logic using the forcing method.

\begin{que}
Is Weak Completeness Theorem for $\mathbf{S4.2}$ provable in $\mathrm{RCA}_0$?
\end{que}

When weak completeness holds for a formula, which is important in the analysis of modal logic, we expect most of them to be provable in $\mathrm{RCA}_0$.

\section*{Acknowledgements}
Takeda is supported by JST SPRING, Grant Number JPMJSP2114. Yokoyama is partially supported by
 JSPS KAKENHI grant numbers JP21KK0045 and JP23K03193.

\bibliography{refmo}
\end{document}